\documentclass[a4paper,11pt]{article}

\usepackage{authblk}
\usepackage{parskip}
\usepackage{mathtools}
\usepackage{amssymb}
\usepackage{amsmath}
\usepackage{latexsym}
\usepackage{mathrsfs}  
\usepackage{bbm}       
\usepackage{amsthm}    
\usepackage{relsize}   
\usepackage{graphicx}
\usepackage{thmtools}  
\usepackage{mathrsfs}  
\usepackage{paralist}  
\usepackage{lipsum}    
\usepackage[all]{xy}   

\numberwithin{equation}{section}

\usepackage{titlesec}   

\titleformat{\section}[block]{\bfseries\filcenter}
{{\upshape\thesection\enspace}}{.5em}{}

\titleformat{\subsection}[block]{\filcenter}
{{\upshape\thesubsection\enspace}}{.5em}{} 

\titleformat{\subsubsection}[block]{\filcenter}
{{\upshape\thesubsubsection\enspace}}{.5em}{} 

\usepackage{enumitem}  
\setlist{nosep}  
\setitemize[0]{leftmargin=*}
\setenumerate[0]{leftmargin=*}
\setenumerate[1]{label={(\arabic*)}} 

\newcommand{\N}{\mathbb{N}}     
\newcommand{\R}{\mathbb{R}}     
\newcommand{\C}{\mathbb{C}}     
\newcommand{\Prob}{\mathbb{P}}  
\newcommand{\Exp}{\mathbb{E}}   
\newcommand{\goth}[1]{\mathfrak{#1}} 
\newcommand{\ind}[2]{\mathbbm{1}_{#1}\left( #2 \right)}          
\newcommand{\inner}[2]{\left\langle #1 \, , \, #2 \right\rangle} 
\newcommand{\norm}[1]{\left|\left|#1\right|\right|}              
\newcommand{\triplet}[3]{\left( #1, #2, #3 \right) }             
\newcommand{\ProbSpace}{\triplet{\Omega}{\mathscr{F}}{\Prob}}    
\newcommand{\abs}[1]{\left| #1 \right|}                          
\newcommand{\defeq}{\mathrel{\mathop:}=}                         

\newcommand\restr[2]{{
  \left.\kern-\nulldelimiterspace 
  #1 
  \vphantom{\big|} 
  \right|_{#2} 
  }}
  
\makeatletter
\newsavebox{\@brx}
\newcommand{\llangle}[1][]{\savebox{\@brx}{\(\m@th{#1\langle}\)}%
  \mathopen{\copy\@brx\kern-0.5\wd\@brx\usebox{\@brx}}}
\newcommand{\rrangle}[1][]{\savebox{\@brx}{\(\m@th{#1\rangle}\)}%
  \mathclose{\copy\@brx\kern-0.5\wd\@brx\usebox{\@brx}}}
\makeatother

\theoremstyle{plain} 
\newtheorem{theorem}{Theorem}[section]    
\newtheorem{proposition}[theorem]{Proposition} 
\newtheorem{corollary}[theorem]{Corollary}
\newtheorem{lemma}[theorem]{Lemma}
\newtheorem{assumption}[theorem]{Assumption}

\theoremstyle{definition} 
\newtheorem{definition}[theorem]{Definition}
\newtheorem{example}[theorem]{Example}
\newtheorem{remark}[theorem]{Remark}

 \title{Stochastic Evolution Equations with L\'{e}vy Noise in the Dual of a Nuclear Space}
\author{C. A. Fonseca-Mora}
\affil{  Escuela de Matem\'{a}tica, Universidad de Costa Rica, San Jos\'{e}, 11501-2060, Costa Rica. \\

\noindent E-mail:  christianandres.fonseca@ucr.ac.cr }

\date{}

\begin{document}

 \maketitle

\abstract{In this article we give sufficient and necessary conditions for the existence of a weak and mild solution to stochastic evolution equations with (general) L\'{e}vy noise taking values in the dual of a nuclear space. As part of our approach we develop a theory of stochastic integration with respect to a L\'{e}vy process taking values in the dual of a nuclear space. 
We also derive further properties of the solution such as the existence of a solution with square moments, the Markov property and path regularity of the solution. In the final part of the paper we give sufficient conditions for the weak convergence of the solutions to a sequence of stochastic evolution equations with L\'{e}vy noises.  }

\smallskip



\section{Introduction}

The main objective of this paper is to study properties of the solutions to the following class of time-inhomogeneous (linear) stochastic evolution equations 
\begin{equation}\label{eqSEELevyIntro}
\begin{cases}
d Y_{t}=A(t)' Y_{t}dt +\int_{\Phi'} R(t,f) L(dt,df), \quad t \geq 0, \\
Y_{0}=\eta, 
\end{cases}
\end{equation}
here $L$ is a L\'{e}vy process taking values in the dual $\Phi'$ to a quasi-complete, bornological nuclear space $\Phi$, $(A(t)':t \geq 0)$ is the family of dual operators to a family $(A(t):t \geq 0)$ of closed linear operators which generates a backward evolution system $(U(s,t): 0 \leq s \leq t < \infty)$ of continuous linear operators on a quasi-complete, bornological nuclear space $\Psi$, and the coefficient $R$ maps $[0,\infty) \times \Omega \times \Phi'$ into the space $\mathcal{L}(\Phi',\Psi')$ of continuous linear operators from $\Phi'$ into $\Psi'$. 

A particular case of \eqref{eqSEELevyIntro} corresponds to 
the following class of (generalized) Langevin equations 
\begin{equation}\label{eqGeneLangEquatIntro}
d Y_{t}=A(t)' Y_{t}dt +B' dL_{t}, \quad t \geq 0.
\end{equation}
here $B' \in \mathcal{L}(\Phi',\Psi')$. Solutions to \eqref{eqGeneLangEquatIntro} are known as (generalized) Ornstein-Uhlenbeck processes. They arise in many investigations on fluctuations limits of infinite particle systems of independent Brownian motions and different types of diffusion systems (see e.g. \cite{BojdeckiGorostiza:1986, DawsonFleischmannGorostiza:1986, HitsudaMitoma:1986, KallianpurMitoma:1992, KallianpurWolpert:1984, Mitoma:1985, Mitoma:1987}). Other applications include the modelling of the dynamics of nerve signals  \cite{KallianpurWolpert:1984}, environmental pollution \cite{KallianpurXiong}, and statistical filtering \cite{KorezliogluMartias:1984}. 

Properties of solutions to stochastic evolution equations of the form \eqref{eqGeneLangEquatIntro} have been studied by many authors (e.g. \cite{BojdeckiGorostiza:1991,  BojdeckiGorostiza:1999, BojdeckiGorostiza:2001,  BojdeckiJakubowski:1999, DawsonGorostiza:1990, FernandezGorostiza:1992, Ito, KallianpurPerezAbreu:1988,  KallianpurPerezAbreu:1989, PerezAbreuTudor:1992, PerezAbreuTudor:1994, Ustunel:1982-2}). In most of these works the authors assumed that $\Phi$ and $\Psi$ are nuclear Fr\'{e}chet spaces, and the driving noise $L$ is a Wiener process or a square integrable martingale, although the case where $L$ is assumed to be a semimartingale with independent increments has also been considered (e.g. see \cite{BojdeckiGorostiza:1991}).

In the case of quasi-complete, bornological nuclear spaces and under the time homogeneous setting $A(t)=A$,  the author introduced in \cite{FonsecaMora:2018-1} sufficient conditions for the existence and uniqueness of weak and mild solutions to non-linear stochastic evolution equations with L\'{e}vy driven multiplicative noise. These sufficient conditions requires that $A$ be the generator of a $(C_{0},1)$-semigroup $(S(t): t \geq 0)$ on $\Psi$ whose dual semigroup $(S(t)': t \geq 0)$ is also a $(C_{0},1)$-semigroup on $\Psi'$ (see definitions in Sect. \ref{subSectEvoSystem}) and the coefficients satisfy some growth and Lipschitz type conditions. However, in the case of linear equations the conditions introduced in \cite{FonsecaMora:2018-1} are rather restrictive and does not cover the time inhomogeneous setting described in \eqref{eqSEELevyIntro}. Further properties of solutions were not studied in \cite{FonsecaMora:2018-1}.  


Motivated by the above, in the work we we present a complete theory of existence and uniqueness of solutions to the general stochastic evolution equation \eqref{eqSEELevyIntro}, we derive some of the fundamental properties of the solution and its trajectories, and we further explore the problem of finding sufficient conditions for the weak convergence of a sequence of stochastic evolution equations of the form \eqref{eqGeneLangEquatIntro}.

We start in Sect. \ref{sectPrelim} with some preliminaries on nuclear spaces and their duals, and basic properties of cylindrical and stochastic processes. In Sect. \ref{sectLevyStochInteg} we introduce a theory of stochastic integration for operator-valued processes with respect to L\'{e}vy processes in duals of nuclear spaces. The stochastic integral is constructed as a sum of the stochastic integrals with respect to each of the components of the L\'{e}vy process according to its L\'{e}vy-It\'{o} decomposition developed in \cite{FonsecaMora:Levy}.   

The stochastic integral with respect to the drift term is defined  using the regularization theorem (see \cite{FonsecaMora:2018}). In the case of the martingale part of the decomposition, the stochastic integral is defined using the theory of stochastic integration with respect to cylindrical martingale-valued measures developed in \cite{FonsecaMora:2018-1}. In the case of the large jumps term in the decomposition the stochastic integral is defined as a finite random sum as in \cite{Applebaum:2006}. The class of stochastic integrands (see Definition \ref{defiIntegrandsLevy}) corresponds to families $R: [0, \infty) \times \Omega \times \Phi' \rightarrow \mathcal{L}(\Phi',\Psi')$ satisfying some weak predictability and weak square integrability conditions in terms of the characteristics of the L\'{e}vy process in its L\'{e}vy-Khintchine formula (see \cite{FonsecaMora:Levy}). We are not aware of any other work that considers stochastic integrals with respect to L\'{e}vy processes in our general context of quasi-complete, bornological nuclear space. 

Our study of existence and uniqueness of ``weak" solutions to \eqref{eqSEELevyIntro} is carried out in Sect. \ref{sectExisteUniqueSEE}. We start by applying our recently introduced theory of stochastic integration to define the stochastic convolution of the forward evolution system $(U(t,s)': 0 \leq s \leq t < \infty)$ (of the dual operators) with respect to the L\'{e}vy process $L$, and by studying some of its properties. The corresponding ``mild'' solution defined by the stochastic convolution (see \eqref{eqDefMildSoluSEELevy}) is latter shown to be weak a solution to \eqref{eqSEELevyIntro} (see Theorem \ref{theoExistSoluSEELevy}). On the other hand, under a mild assumption on the regularity of the solutions (that holds for example if the solution has finite moments or c\`{a}dl\`{a}g paths) it is shown in Theorem \ref{theoUniqueSoluSEELevy} that any weak solution to \eqref{eqSEELevyIntro} is a version of the mild solution. Here it is worth to mention that unlike other works that establish existence and uniqueness of solutions to equations of the form \eqref{eqGeneLangEquatIntro} (e.g. \cite{BojdeckiGorostiza:1991, KallianpurPerezAbreu:1988}), we do not require for the family of generators $(A(t): t \geq 0)$ that each $A(t)$ be a continuous and linear operator on $\Psi$.

In Sect. \ref{sectPropertiesSolution} we study further properties of the solutions to \eqref{eqSEELevyIntro}. We start by considering the existence of solutions with square moments under the assumption that $L$ has square moments. In particular, in Theorem \ref{theoExisUniqStronSquaIntegSolution} we show that there exists a unique solution $(X_{t}: t \geq 0)$ to  \eqref{eqSEELevyIntro} which has the property that for each $t>0$ there exists a Hilbert space $\Psi'_{\varrho}$ embedded in $\Psi$, $\Psi'_{\varrho}$ equipped with a Hilbertian norm $\varrho'$, such that $\Exp \int_{0}^{t} \varrho'(X_{s})^{2} ds < \infty$. In a second part of this section we will concentrate our efforts in the study of equations of the form \eqref{eqGeneLangEquatIntro}, where we prove that the solution is a Markov process (Proposition \ref{propOrnsUhleIsMarkov}). Moreover, we will show  in Theorem \ref{theoCadlagVersiContiGenerator}
that under the hypothesis that each $A(t)$ is a continuous and linear operator on $\Psi$ there exists a unique c\`{a}dl\`{a}g solution to \eqref{eqGeneLangEquatIntro}. 

Examples and applications of our results will be given in Sect. \ref{secExamApplica}. In particular, we will see that our results generalize to the L\'{e}vy noise setting those results obtained in \cite{KallianpurPerezAbreu:1988} for stochastic equations driven by square-integrable martingale noise in the dual of a Fr\'{e}chet nuclear space (see Example \ref{examplePertubEquations} for details). 

Finally, in Sect. \ref{weakConvSEE} we provide sufficient conditions for the weak convergence of the   solutions to the sequence of (generalized) Langevin equations
$$ dY^{n}_{t}=A^{n}(t)'Y^{n} dt+ (B^{n})' dL^{n}_{t}, \quad Y^{n}_{0}=\eta^{n}_{0}. $$
Our main result (Theorem \ref{theoWeakConvOrnsUhleProce}) generalize those obtained by other authors (see \cite{FernandezGorostiza:1992, KallianpurPerezAbreu:1989, PerezAbreuTudor:1992}) from the context of Fr\'{e}chet spaces to complete, bornological nuclear spaces. Furthermore, we also formulate (see Theorem \ref{theoWeakConvOrnsUhleProceVerCharac}) sufficient conditions for the weak convergence of solutions in terms of convergence on finite-dimensional distributions of the sequence of L\'{e}vy processes $L^{n}$ and some properties of the sequence of the characteristics of each $L^{n}$ in its L\'{e}vy-Khintchine formula.

\section{Preliminaries}\label{sectPrelim}

Let $\Phi$ be a locally convex space (we will only consider vector spaces over $\R$). $\Phi$ is \emph{quasi-complete} if each bounded and closed subset of it is complete. $\Phi$ is called  \emph{bornological} (respectively \emph{ultrabornological}) if it is the inductive limit of a family of normed (respectively Banach)  spaces. A \emph{barreled space} is a locally convex space such that every convex, balanced, absorbing and closed subset is a neighborhood of zero. Every quasi-complete bornological space is ultrabornological and hence barrelled. For further details see \cite{Jarchow, NariciBeckenstein}.   
 
If $p$ is a continuous semi-norm on $\Phi$ and $r>0$, the closed ball of radius $r$ of $p$ given by $B_{p}(r) = \left\{ \phi \in \Phi: p(\phi) \leq r \right\}$ is a closed, convex, balanced neighborhood of zero in $\Phi$. A continuous semi-norm (respectively a norm) $p$ on $\Phi$ is called \emph{Hilbertian} if $p(\phi)^{2}=Q(\phi,\phi)$, for all $\phi \in \Phi$, where $Q$ is a symmetric, non-negative bilinear form (respectively inner product) on $\Phi \times \Phi$. Let $\Phi_{p}$ be the Hilbert space that corresponds to the completion of the pre-Hilbert space $(\Phi / \mbox{ker}(p), \tilde{p})$, where $\tilde{p}(\phi+\mbox{ker}(p))=p(\phi)$ for each $\phi \in \Phi$. The quotient map $\Phi \rightarrow \Phi / \mbox{ker}(p)$ has an unique continuous linear extension $i_{p}:\Phi \rightarrow \Phi_{p}$.  Let $q$ be another continuous Hilbertian semi-norm on $\Phi$ for which $p \leq q$. In this case, $\mbox{ker}(q) \subseteq \mbox{ker}(p)$. Moreover, the inclusion map from $\Phi / \mbox{ker}(q)$ into $\Phi / \mbox{ker}(p)$ is linear and continuous, and therefore it has a unique continuous extension $i_{p,q}:\Phi_{q} \rightarrow \Phi_{p}$. Furthermore, we have the following relation: $i_{p}=i_{p,q} \circ i_{q}$. 

Let $p$ and $q$ be continuous Hilbertian semi-norms on $\Phi$ such that $p \leq q$.
The space of continuous linear operators (respectively Hilbert-Schmidt operators) from $\Phi_{q}$ into $\Phi_{p}$ is denoted by $\mathcal{L}(\Phi_{q},\Phi_{p})$ (respectively $\mathcal{L}_{2}(\Phi_{q},\Phi_{p})$) and the operator norm (respectively Hilbert-Schmidt norm) is denote by $\norm{\cdot}_{\mathcal{L}(\Phi_{q},\Phi_{p})}$ (respectively $\norm{\cdot}_{\mathcal{L}_{2}(\Phi_{q},\Phi_{p})}$).

We denote by $\Phi'$ the topological dual of $\Phi$ and by $\inner{f}{\phi}$ the canonical pairing of elements $f \in \Phi'$, $\phi \in \Phi$. Unless otherwise specified, $\Phi'$ will always be consider equipped with its \emph{strong topology}, i.e. the topology on $\Phi'$ generated by the family of semi-norms $( \eta_{B} )$, where for each $B \subseteq \Phi$ bounded we have $\eta_{B}(f)=\sup \{ \abs{\inner{f}{\phi}}: \phi \in B \}$ for all $f \in \Phi'$.  If $p$ is a continuous Hilbertian semi-norm on $\Phi$, then we denote by $\Phi'_{p}$ the Hilbert space dual to $\Phi_{p}$. The dual norm $p'$ on $\Phi'_{p}$ is given by $p'(f)=\sup \{ \abs{\inner{f}{\phi}}:  \phi \in B_{p}(1) \}$ for all $ f \in \Phi'_{p}$. Moreover, the dual operator $i_{p}'$ corresponds to the canonical inclusion from $\Phi'_{p}$ into $\Phi'$ and it is linear and continuous.


Let $( q_{\gamma}(\cdot): \gamma \in \Gamma )$ be a family of seminorms generating the strong topology on $\Phi'$. Fix $T>0$ and denote by $D_{T}(\Phi')$ the collection of all c\`{a}dl\`{a}g (i.e. right-continuous with left limits) maps from $[0,T]$ into $\Phi'$. For a given $ \gamma \in \Gamma$ we consider the pseudometric $d_{\gamma}$ on $D_{T}(\Phi')$ given by  
\begin{equation}\label{defSkorokhodPseudometrics}
d_{\gamma}(x,y)=\inf_{\lambda \in \Lambda}  \left\{ \sup_{t \in [0,T]} q_{\gamma}(x(t)-y(\lambda(t))) + \sup_{0 \leq s< t \leq T} \abs{\log \frac{\lambda(t)-\lambda(s)}{t-s}} \right\},
\end{equation} 
for all $x, y \in D_{T}(\Phi')$, where $\Lambda$ denotes the set of all the strictly increasing continuous maps $\lambda$ from $[0,T]$ onto itself. The family of seminorms $(d_{\gamma}: \gamma \in \Gamma )$ generates a completely regular topology on $D_{T}(\Phi')$ that is known as the \emph{Skorokhod  topology} (also known as the $J1$ topology).

Let $D_{\infty}(\Phi')$  denote the space of mappings $x: [0,\infty) \rightarrow \Phi'_{\beta}$ wich are c\`{a}dl\`{a}g. For every $s \geq 0$, let $r_{s}: D(\Phi'_{\beta}) \rightarrow D_{s+1}(\Phi'_{\beta})$  be given by
$$r_{s}(x)(t) =
\begin{cases}
x(t) & \mbox{ if } t \in [0,s],\\
(s+1-t)x(t) & \mbox{ if } t \in [s,s+1].
\end{cases}
$$
For every $\gamma \in \Gamma$ let
$$d^{\infty}_{\gamma}(x,y)=\sum_{n =1}^{\infty} \frac{1}{2^{n}} \left( 1 \wedge d^{n}_{\gamma}(r_{n}(x),r_{n}(y)) \right),$$
where for each $n \in \N$, $d^{n}_{\gamma}$ is the pseudometric defined in  \eqref{defSkorokhodPseudometrics} for $T=n$.  The \emph{Skorokhod topology} in $D_{\infty}(\Phi')$ is a completely regular topology generated by the family of pseudometrics $( d^{\infty}_{\gamma}: \gamma \in \Gamma)$. For further details see \cite{FonsecaMora:Skorokhod, Jakubowski:1986}.

Let us recall that a (Hausdorff) locally convex space $(\Phi,\mathcal{T})$ is called \emph{nuclear} if its topology $\mathcal{T}$ is generated by a family $\Pi$ of Hilbertian semi-norms such that for each $p \in \Pi$ there exists $q \in \Pi$, satisfying $p \leq q$ and the canonical inclusion $i_{p,q}: \Phi_{q} \rightarrow \Phi_{p}$ is Hilbert-Schmidt. Other equivalent definitions of nuclear spaces can be found in \cite{Pietsch, Treves}. 

Let $\Phi$ be a nuclear space. If $p$ is a continuous Hilbertian semi-norm  on $\Phi$, then the Hilbert space $\Phi_{p}$ is separable (see \cite{Pietsch}, Proposition 4.4.9 and Theorem 4.4.10, p.82). Now, let $( p_{n} : n \in \N)$ be an increasing sequence of continuous Hilbertian semi-norms on $(\Phi,\mathcal{T})$. We denote by $\theta$ the locally convex topology on $\Phi$ generated by the family $( p_{n} : n \in \N)$. The topology $\theta$ is weaker than $\mathcal{T}$. We  will call $\theta$ a (weaker) \emph{countably Hilbertian topology} on $\Phi$ and we denote by $\Phi_{\theta}$ the space $(\Phi,\theta)$ and by $\widehat{\Phi}_{\theta}$ its completion. The space $\widehat{\Phi}_{\theta}$ is a (not necessarily Hausdorff) separable, complete, pseudo-metrizable (hence Baire and ultrabornological; see Example 13.2.8(b) and Theorem 13.2.12 in \cite{NariciBeckenstein}) locally convex space and its dual space satisfies $(\widehat{\Phi}_{\theta})'=(\Phi_{\theta})'=\bigcup_{n \in \N} \Phi'_{p_{n}}$ (see \cite{FonsecaMora:2018}, Proposition 2.4). 

The following are all examples of complete, ultrabornological (hence barrelled) nuclear spaces: 
the spaces of functions $\mathscr{E}_{K} \defeq \mathcal{C}^{\infty}(K)$ ($K$: compact subset of $\R^{d}$) and $\mathscr{E}\defeq \mathcal{C}^{\infty}(\R^{d})$, the rapidly decreasing functions $\mathscr{S}(\R^{d})$, and the space of test functions $\mathscr{D}(U) \defeq \mathcal{C}_{c}^{\infty}(U)$ ($U$: open subset of $\R^{d}$), as well are the spaces of distributions $\mathscr{E}'_{K}$, $\mathscr{E}'$, $\mathscr{S}'(\R^{d})$, and $\mathscr{D}'(U)$. Other examples are the space of harmonic functions $\mathcal{H}(U)$ ($U$: open subset of $\R^{d}$), the space of polynomials $\mathcal{P}_{n}$ in $n$-variables and the space of real-valued sequences $\R^{\N}$ (with direct product topology). For references see \cite{Pietsch, Schaefer, Treves}.


Throughout this work we assume that $\ProbSpace$ is a complete probability space and consider a filtration $( \mathcal{F}_{t} : t \geq 0)$ on $\ProbSpace$ that satisfies the \emph{usual conditions}, i.e. it is right continuous and $\mathcal{F}_{0}$ contains all subsets of sets of $\mathcal{F}$ of $\Prob$-measure zero. We denote by $L^{0} \ProbSpace$ the space of equivalence classes of real-valued random variables defined on $\ProbSpace$. We always consider the space $L^{0} \ProbSpace$ equipped with the topology of convergence in probability and in this case it is a complete, metrizable, topological vector space. 
We denote by $\mathcal{P}_{\infty}$ the predictable $\sigma$-algebra on $[0, \infty) \times \Omega$ and for any $T>0$, we denote by $\mathcal{P}_{T}$ the restriction of $\mathcal{P}_{\infty}$ to $[0,T] \times \Omega$.  

A \emph{cylindrical random variable}\index{cylindrical random variable} in $\Phi'$ is a linear map $X: \Phi \rightarrow L^{0} \ProbSpace$ (see \cite{FonsecaMora:2018}). If $X$ is a cylindrical random variable in $\Phi'$, we say that $X$ is \emph{$n$-integrable} ($n \in \N$)  if $ \Exp \left( \abs{X(\phi)}^{n} \right)< \infty$, $\forall \, \phi \in \Phi$, and has \emph{zero-mean} if $ \Exp \left( X(\phi) \right)=0$, $\forall \phi \in \Phi$. The \emph{Fourier transform} of $X$ is the map from $\Phi$ into $\C$ given by $\phi \mapsto \Exp ( e^{i X(\phi)})$.

Let $X$ be a $\Phi'$-valued random variable, i.e. $X:\Omega \rightarrow \Phi'_{\beta}$ is a $\mathscr{F}/\mathcal{B}(\Phi'_{\beta})$-measurable map. For each $\phi \in \Phi$ we denote by $\inner{X}{\phi}$ the real-valued random variable defined by $\inner{X}{\phi}(\omega) \defeq \inner{X(\omega)}{\phi}$, for all $\omega \in \Omega$. The linear mapping $\phi \mapsto \inner{X}{\phi}$ is called the \emph{cylindrical random variable induced/defined by} $X$. We will say that a $\Phi'$-valued random variable $X$ is \emph{$n$-integrable} ($n \in \N$) if the cylindrical random variable induced by $X$ is \emph{$n$-integrable}. 
 
Let $J=\R_{+} \defeq [0,\infty)$ or $J=[0,T]$ for  $T>0$. We say that $X=( X_{t}: t \in J)$ is a \emph{cylindrical process} in $\Phi'$ if $X_{t}$ is a cylindrical random variable for each $t \in J$. Clearly, any $\Phi'$-valued stochastic processes $X=( X_{t}: t \in J)$ induces/defines a cylindrical process under the prescription: $\inner{X}{\phi}=( \inner{X_{t}}{\phi}: t \in J)$, for each $\phi \in \Phi$. 

If $X$ is a cylindrical random variable in $\Phi'$, a $\Phi'$-valued random variable $Y$ is called a \emph{version} of $X$ if for every $\phi \in \Phi$, $X(\phi)=\inner{Y}{\phi}$ $\Prob$-a.e. A $\Phi'$-valued process $Y=(Y_{t}:t \in J)$ is said to be a $\Phi'$-valued \emph{version} of the cylindrical process $X=(X_{t}: t \in J)$ on $\Phi'$ if for each $t \in J$, $Y_{t}$ is a $\Phi'$-valued version of $X_{t}$.  

For a $\Phi'$-valued process $X=( X_{t}: t \in J)$ terms like continuous, c\`{a}dl\`{a}g, purely discontinuous, adapted, predictable, etc. have the usual (obvious) meaning. 

A $\Phi'$-valued random variable $X$ is called \emph{regular} if there exists a weaker countably Hilbertian topology $\theta$ on $\Phi$ such that $\Prob( \omega: X(\omega) \in (\widehat{\Phi}_{\theta})')=1$. If $\Phi$ is barrelled, the property of being regular is  equivalent to the property that the law of $Y$ is a Radon measure on $\Phi'$ (see Theorem 2.10 in \cite{FonsecaMora:2018}). A $\Phi'$-valued process $Y=(Y_{t}:t \in J)$ is said to be \emph{regular} if $Y_{t}$ is a regular random variable for each $t \in J$.

\begin{assumption} All through this article $\Phi$ and $\Psi$ will denote two quasi-complete, bornological, nuclear spaces.
\end{assumption}

\section{L\'{e}vy Processes and Stochastic Integrals} \label{sectLevyStochInteg}

Our main objective in this section is to review some properties on L\'{e}vy processes in duals of nuclear spaces and to define stochastic integrals with respect to these processes. 

\subsection{L\'{e}vy processes in duals of nuclear spaces}
\label{subSectLevyDualNuclear}

Recall from \cite{FonsecaMora:Levy} that a $\Phi'$-valued process $L=( L_{t} : t\geq 0)$ is called a \emph{L\'{e}vy process} if \begin{inparaenum}[(i)] \item  $L_{0}=0$ a.s., 
\item $L$ has \emph{independent increments}, i.e. for any $n \in \N$, $0 \leq t_{1}< t_{2} < \dots < t_{n} < \infty$ the $\Phi'$-valued random variables $L_{t_{1}},L_{t_{2}}-L_{t_{1}}, \dots, L_{t_{n}}-L_{t_{n-1}}$ are independent,  
\item L has \emph{stationary increments}, i.e. for any $0 \leq s \leq t$, $L_{t}-L_{s}$ and $L_{t-s}$ are identically distributed, and  
\item For every $t \geq 0$ the distribution $\mu_{t}$ of $L_{t}$ is a Radon measure and the mapping $t \mapsto \mu_{t}$ from $\R_{+}$ into the space $\goth{M}_{R}^{1}(\Phi')$ of Radon probability measures on $\Phi'$ is continuous at $0$ when $\goth{M}_{R}^{1}(\Phi')$  is equipped with the weak topology. \end{inparaenum} It is shown in Corollary 3.11 in \cite{FonsecaMora:Levy} that $L=( L_{t} : t\geq 0)$ has a regular, c\`{a}dl\`{a}g version $\tilde{L}=( \tilde{L}_{t} : t \geq 0)$ that is also a L\'{e}vy process. Moreover, there exists a weaker countably Hilbertian topology $\vartheta$ on $\Phi$ such that $\tilde{L}$ is a $(\widehat{\Phi}_{\vartheta})'$-valued c\`{a}dl\`{a}g process. We will therefore identify $L$ with $\tilde{L}$.

Let $N=\{N(t,A): \, t \geq 0, A \in \mathcal{B}(\Phi' \setminus \{ 0\})\}$ be the \emph{Poisson random measure} associated to $L$, i.e. $ N(t,A)= \sum_{0 \leq s \leq t} \ind{A}{\Delta L_{s}}$, $\forall \, t \geq 0$, $A \in \mathcal{B}( \Phi' \setminus \{ 0\})$, with respect to the ring $\mathcal{A}$ of all the subsets of $\Phi' \setminus \{0\}$ that are \emph{bounded below} (i.e. $A \in \mathcal{A}$ if $0 \notin \overline{A}$, where $\overline{A}$ is the closure of $A$). The corresponding compensator measure of $N$ is of the form $\nu(\omega; dt; df)= dt \nu (df)$, where $\nu$ is a \emph{L\'{e}vy measure} on $\Phi'$ in the following sense (see \cite{FonsecaMora:Levy}, Theorem 4.11):
\begin{enumerate}
\item $\nu (\{ 0 \})=0$, 
\item for each neighborhood of zero $U \subseteq \Phi'$, the  restriction $\restr{\nu}{U^{c}}$ of $\nu$ on the set $U^{c}$ belongs to the space $\goth{M}^{b}_{R}(\Phi')$ of bounded Radon measures on $\Phi'$,    
\item there exists a continuous Hilbertian semi-norm $\rho$ on $\Phi$ such that 
\begin{equation} \label{integrabilityPropertyLevyMeasure}
\int_{B_{\rho'}(1)} \rho'(f)^{2} \nu (df) < \infty,  \quad \mbox{and} \quad  \restr{\nu}{B_{\rho'}(1)^{c}} \in \goth{M}^{b}_{R}(\Phi'), 
\end{equation}
where  $\rho'$ is the dual norm of $\rho$ and $B_{\rho'}(1) \defeq B_{\rho}(1)^{0}=\{ f \in \Phi': \rho'(f) \leq 1\}$. 
\end{enumerate}
Here is important to stress the fact that the seminorm $\rho$ satisfying \eqref{integrabilityPropertyLevyMeasure} is not unique. Indeed, any continuous Hilbert semi-norm $q$ on  such that $\rho \leq q$ satisfies \eqref{integrabilityPropertyLevyMeasure}.

It is shown in Theorem 4.17 in \cite{FonsecaMora:Levy} that relative to a continuous Hilbertian seminorm $\rho$ on $\Phi$ satisfying \eqref{integrabilityPropertyLevyMeasure},  for each $t \geq 0$, $L_{t}$ admits the unique representation
\begin{equation} \label{eqLevyItoDecomposition}
L_{t}=t\goth{m}+W_{t}+\int_{B_{\rho'}(1)} f \widetilde{N} (t,df)+\int_{B_{\rho'}(1)^{c}} f N (t,df)
\end{equation}
that is usually called the \emph{L\'{e}vy-It\^{o} decomposition} of $L$. In \eqref{eqLevyItoDecomposition}, we have that $\goth{m} \in \Phi'$, $\widetilde{N}(dt,df)= N(dt,df)-dt \, \nu(df)$ is the compensated Poisson random measure, and $( W_{t} : t \geq 0)$ is a $\Phi'$-valued L\'{e}vy process with continuous paths (also called a $\Phi'$-valued Wiener process) with zero-mean (i.e. $\Exp (\inner{W_{t}}{\phi}) = 0$ for each $t \geq 0$ and $\phi \in \Phi$) and \emph{covariance functional} $\mathcal{Q}$ satisfying 
\begin{equation}\label{covarianceFunctWienerProcess}
\Exp \left( \inner{W_{t}}{\phi} \inner{W_{s}}{\varphi} \right) = ( t \wedge s ) \mathcal{Q} (\phi, \varphi), \quad \forall \, \phi, \varphi \in \Phi, \, s, t \geq 0. 
\end{equation}
Observe that $\mathcal{Q}$ is a continuous, symmetric, non-negative bilinear form on $\Phi \times \Phi$. The associated Hilbertian seminorm defined by $\mathcal{Q}(\cdot, \cdot)$ will be denoted by $\mathcal{Q}(\cdot)$. The process $\int_{B_{\rho'}(1)} f \widetilde{N} (t,df)$, $t\geq 0$, is a $\Phi'$-valued zero-mean, square integrable, c\`{a}dl\`{a}g L\'{e}vy process with second moments given by $\Exp \left( \abs{ \inner{\int_{B_{\rho'}(1)} f \widetilde{N} (t,df)}{\phi} }^{2}\right) = t \int_{B_{\rho'}(1)} \abs{\inner{f}{\phi}}^{2} \nu (df)$ $\forall \, t \geq 0$ and $\phi \in \Phi$, and the process $\int_{B_{\rho'}(1)^{c}} f N (t,df)$  $\forall t\geq 0$ is a $\Phi'$-valued c\`{a}dl\`{a}g L\'{e}vy process defined  by means of a Poisson integral with respect to the Poisson random measure $N$ of $L$ on the set $B_{\rho'}(1)^{c}$ (see Section 4.1 in \cite{FonsecaMora:Levy} for more information on Poisson integrals in duals of nuclear spaces). It is important to remark that all the random components of the representation \eqref{eqLevyItoDecomposition} are independent.

The Fourier transform of the L\'{e}vy process $L=( L_{t}: t\geq 0)$ is given by the \emph{L\'{e}vy-Khintchine formula} which characterizes it uniquely (see \cite{FonsecaMora:Levy}, Theorem 4.18): for each $t \geq 0$, $\phi \in \Phi$, 
\begin{equation} \label{levyKhintchineFormulaLevyProcess}
\begin{split}
& \Exp \left( e^{i \inner{L_{t}}{\phi} } \right) = e^{t \xi(\phi)}, \quad  \mbox{ with} \\ 
& \xi(\phi)= i \inner{\goth{m}}{\phi} - \frac{1}{2} \mathcal{Q}(\phi)^{2} + \int_{\Phi'} \left( e^{i \inner{f}{\phi}} -1 - i \inner{f}{\phi} \ind{ B_{\rho'}(1)}{f} \right) \nu(d f).  
\end{split}
\end{equation}
where the \emph{characteristics} $(\goth{m}, \mathcal{Q},  \nu, \rho)$ of $L$ are as described above.

\subsection{Stochastic Integration}\label{subSectStochIntegLevy}

Our next objective is to define stochastic integrals with respect to the $\Phi'$-valued L\'{e}vy process $L=(L_{t}: t \geq 0)$. Our plan is to define the stochastic integral of $L$ as a sum of stochastic integrals with respect to 
each term in the L\'{e}vy-It\^{o} decomposition \eqref{eqLevyItoDecomposition}. We show the existence and properties of these stochastic integrals in the following paragraphs. 

Let $R: [0, \infty) \times \Omega \times \Phi' \rightarrow \mathcal{L}(\Phi',\Psi')$ satisfying that for every $\psi \in \Psi$, $T>0$,  the mapping $(t,\omega,f) \mapsto R(t,\omega,f)'\psi$ is $\mathcal{P}_{T} \otimes \mathcal{B}(\Phi')/\mathcal{B}(\Phi)$-measurable. 

\begin{enumerate}
\item \textbf{Drift stochastic integral:} Let $h:[0, \infty) \times \Omega \rightarrow \Phi'$ such that for every $T>0$ the mapping $(t,\omega) \mapsto h(t,\omega)$ is $\mathcal{P}_{T}/\mathcal{B}(\Phi')$-measurable.
Assume moreover that 
$$ \Prob \left( \omega \in \Omega:  \int_{0}^{T} \, \abs{\inner{h(r, \omega)}{R(r, \omega,0)'\psi}} \, dr < \infty \right)=1, \quad \forall \, T>0, \, \psi \in \Psi. $$
Then there exists a $\Phi'$-valued, regular, continuous process $\int_{0}^{t} R(r,0) h(r) dr$, $t \geq 0$, satisfying $\Prob$-a.e.
\begin{equation}\label{eqDefiDetermiIntegDriftGeneral}
\inner{\int_{0}^{t} R(r,\omega,0)h(r,\omega) dr}{\psi}= \int_{0}^{t} \inner{h(r,\omega)}{R(r,\omega,0)'\psi} ds, \quad \forall \, t \geq 0,  \omega \in \Omega, \psi \in \Psi. 
\end{equation}
If we furthermore have for some $n \geq 1$ that 
\begin{equation}\label{eqFiniteNMomentDriftGeneral}
\Exp \int_{0}^{T} \, \abs{\inner{h(r)}{R(r,0)'\psi}}^{n} \, dr < \infty , \quad \forall \psi \in \Psi, 
\end{equation}
then $\int_{0}^{t} R(r,0) h(r) dr$, $t \geq 0$, is $n$-integrable. 

We will show that such a process exists. 
To do this we will need the following information on absolutely continuous functions. For $t>0$, let $A C_{t}$ denotes the linear space of all absolutely continuous functions on $[0,t]$ which are zero at $0$.
It is well-known (see Theorem 5.3.6 in \cite{BogachevMT}, p.339) that $G \in A C_{t}$ if and only if there exists an integrable function $g$ defined on $[0,t]$ such that:
\begin{equation} \label{defDensityAbsContFunct}
G(s)= \int^{s}_{0} g(r) dr, \quad \forall s \in [0,t].
\end{equation}
The space $A C_{t}$ is a Banach space equipped with the norm $\norm{\cdot}_{A C_{t}}$ given by $\norm{G}_{A C_{t}}= \int^{t}_{0} \abs{g(r)} dr$,  for $G \in A C_{t}$ with $g$ satisfying \eqref{defDensityAbsContFunct}.

Let $X_{0}(\psi)=0$ for every $\psi \in \Psi$,  and for each $t>0$, define $X_{t}: \Psi \rightarrow L^{0}(\Omega, \mathcal{F}, \Prob; AC_{t}) $ by 
$$X_{t}(\psi)(\omega)(s)=
\begin{cases}
\int_{0}^{s} \inner{h(r,\omega)}{R(r,\omega,0)'\psi} dr, & \mbox{for }  s \in [0,t], \, \omega \in \Omega_{t,\psi}, \\
0, & \mbox{elsewhere},
\end{cases}
$$
for every $\psi \in \Psi$, where $\Omega_{t,\psi}=\{ \omega \in \Omega:  \int_{0}^{T} \, \abs{\inner{h(r,\omega)}{R(r,\omega,0)'\psi}} \, dr < \infty\}$. The mapping $X_{t}$ is  linear. We will show it is sequentially closed, where $L^{0}(\Omega, \mathcal{F}, \Prob; AC_{t})$ is equipped with the topology of convergence in probability in the norm $\norm{\cdot}_{A C_{t}}$. Suppose $\psi_{n} \rightarrow \psi$ and $X_{t}(\psi_{n}) \rightarrow G$ in $L^{0}(\Omega, \mathcal{F}, \Prob; AC_{t})$. There exists a subsequence $(\psi_{n_{k}})$ such that $\psi_{n_{k}} \rightarrow \psi$ and $X_{t}(\psi_{n_{k}}) \rightarrow G$ in $AC_{t}$ $\Prob$-a.e. Since $G \in L^{0}(\Omega, \mathcal{F}, \Prob; AC_{t})$, there exists a stochastic process $g(t,\omega)$ satisfying \eqref{defDensityAbsContFunct} for every $\omega \in \Omega$. Since $
\inner{h(r,\omega)}{R(r, \omega,0)'\psi_{n_{k}}} \rightarrow \inner{h(r,\omega,0)}{R(r,\omega)'\psi}$ for all $r \in [0,s]$, $\omega \in \Omega$, then by Fatou's lemma we have $\Prob$-a.e.
\begin{multline*}
\norm{X_{t}(\psi)(\omega)-G(\omega)}_{AC_{t}} =  \int^{t}_{0} \lim_{k \rightarrow \infty} \abs{\inner{h(r,\omega)}{R(r,\omega,0)'\psi_{n_{k}}}-g(r,\omega)} dr \\ 
\leq  \liminf_{k \rightarrow \infty}  \int^{t}_{0} \abs{\inner{h(r,\omega)}{R(r,\omega,0)'\psi_{n_{k}}}-g(r,\omega)} dr= \lim_{k \rightarrow \infty}  \norm{X_{t}(\psi_{n_{k}})(\omega)-G(\omega)}_{AC_{t}}=0.
\end{multline*}
Since $\Psi$ is ultrabornological and $L^{0}(\Omega, \mathcal{F}, \Prob; AC_{t})$ is a complete, metrizable, topological vector space, the closed graph theorem (Theorem 14.7.3 in \cite{NariciBeckenstein}, p.475) shows that $X_{t}$ is continuous. Since the projection mapping $\pi_{t}: L^{0}(\Omega, \mathcal{F}, \Prob; AC_{t}) \rightarrow L^{0}(\Omega, \mathcal{F}, \Prob)$, $\pi_{t}(Y)(\omega)=Y(\omega)(t)$ $\forall \omega \in \Omega$, is linear and continuous, then $(\pi_{t} \circ X_{t}: t \geq 0)$ defines a cylindrical process in $\Psi'$ such that for each $t \geq 0$, $\pi_{t} \circ X_{t}: \Psi \rightarrow  L^{0}(\Omega, \mathcal{F}, \Prob)$ is continuous. Hence, as for each $\psi \in \Psi$, $t \mapsto X_{t}(\psi)$ is continuous, the regularization theorem for ultrabornological spaces (Corollary 3.11 in \cite{FonsecaMora:2018}) shows that there exists a $\Psi'$-valued regular continuous process $\int_{0}^{t} R(r,0) h(r) dr$, $t \geq 0$, that is a version of $(\pi_{t} \circ X_{t}: t \geq 0)$, i.e. satisfying 
\eqref{eqDefiDetermiIntegDriftGeneral}. 

Finally, the fact that $\int_{0}^{t} R(r,0) h(r) dr$, $t \geq 0$, is $n$-integrable whenever \eqref{eqFiniteNMomentDriftGeneral} is satisfied is a direct consequence of \eqref{eqDefiDetermiIntegDriftGeneral} and Jensen's inequality. 

\item \textbf{Wiener and compensated Poisson stochastic integrals:} 
If we assume that for every $T >0$ and $\psi \in \Psi$ the mapping $R$ satisfies 
\begin{equation*}
\Exp \int_{0}^{T} \, \mathcal{Q}(R(r,0)'\psi)^{2} \, dr< \infty, 
\end{equation*}  
then we can define the Wiener stochastic integral $\int_{0}^{t}  R(r,0) dW_{r}$, $t \geq 0$, using the 
the theory of (strong) stochastic integration developed in Section 5 in \cite{FonsecaMora:2018-1} with respect to the martingale valued measure 
$$M_{1}(t,A)=W_{t} \delta_{0}(A), \quad \mbox{ for } \, t \geq 0, \, A \in \mathcal{B}(\{0\}),$$
(see Example 3.4 in \cite{FonsecaMora:2018-1}). The Wiener stochastic integral is a $\Psi'$-valued zero-mean, square integrable martingale with continuous paths.  

Let $U \in \mathcal{B}(\Phi')$ be such that $\int_{U} \abs{\inner{f}{\phi}}^{2} \nu(df) < \infty$ for all $\phi \in \Phi$. If we assume that for every $T >0$ and $\psi \in \Psi$ the mapping $R$ satisfies 
\begin{equation*}
\Exp \int_{0}^{T} \, \int_{U} \abs{\inner{f}{R(r,f)'\psi}}^{2} \nu(df)\,  \, dr< \infty, 
\end{equation*}  
then we can define the compensated Poisson stochastic integral on $U$, $\int_{0}^{t} \int_{U} R(r,f) \widetilde{N}(dr,df)$, $t \geq 0$, using the 
the theory of (strong) stochastic integration developed in Section 5 in \cite{FonsecaMora:2018-1} with respect to the martingale valued measure 
\begin{equation}\label{eqDefiPoissonMartValuMea}
M_{2}(t,A)=\int_{A} f \widetilde{N}(t,df), \quad \mbox{ for } \, t \geq 0, \, A \in \mathcal{A} \cap U,
\end{equation}
(see Example 3.5 in \cite{FonsecaMora:2018-1}). 
The compensated Poisson stochastic integral on $U$ is a $\Psi'$-valued zero-mean, square integrable c\`{a}dl\`{a}g martingale. 

For all $t \geq 0$,  let
$$ \int^{t}_{0} \int_{U} R (r,f) M (dr, df) \defeq \int_{0}^{t}  R(s,0) dW_{s} + \int_{0}^{t} \int_{U}  R(s,f) \widetilde{N}(ds,du). $$
A simple application of Proposition 4.12 in  \cite{FonsecaMora:2018-1} shows that $\int^{t}_{0} \int_{U} R (r,f) M (dr, df)$ can be defined equivalently as the stochastic integral of $R$ with respect to the L\'{e}vy martingale-valued measure 
\begin{equation} \label{levyMartValuedMeasExam} 
M(t,A) = W_{t} \delta_{0}(A) + \int_{A \backslash \{0 \}} f \widetilde{N}(t,df), \quad \mbox{ for } \, t \geq 0, \, A \in \mathcal{R}, 
\end{equation}
where $\mathcal{R}= \{  U \cap \Gamma: \Gamma \in \mathcal{A} \} \cup \{ \{0\}\}$. 

For every $\psi \in \Psi$, $\Prob$-a.e. $\forall t \in [0,T]$, we have the following \emph{weak-strong compatibility}:
\begin{equation}\label{eqWeakStrongCompLevyMartValMeas}
 \inner{\int^{t}_{0} \int_{U} R (r,f) M (dr, df)}{\psi}=\int^{t}_{0} \int_{U} R (r,f)'\psi M (dr, df),
\end{equation}
where the real-valued (weak) stochastic integral on the right-hand side of the above equality is defined by Theorem 4.7 in \cite{FonsecaMora:2018-1}.  


\item \textbf{Poisson stochastic integral:} 
Assume that $V \in \mathcal{A}$, i.e. $0 \notin \overline{V}$. For each $t \geq 0$ we define the Poisson stochastic integral:
\begin{equation} \label{eqDefiPoissonInteg}
\int_{0}^{t} \int_{V} R(s,f)  N(ds,df)(\omega) 
= \sum_{0 \leq s \leq t} R(s, \omega, \Delta L_{s}(\omega)) \Delta L_{s}(\omega) \ind{V}{\Delta L_{s}(\omega)},
\end{equation}
which is a finite (random) sum since for $\Prob$-a.e. $\omega \in \Omega$ the trajectories $t \mapsto L_{t}(\omega)$ only have a finite number of jumps on the time integral $[0,t]$ and $0 \notin \overline{V}$ (see the discussion in Section 4.2 in \cite{FonsecaMora:Levy}). By definition, the Poisson stochastic integral \eqref{eqDefiPoissonInteg} is a $\Psi'$-valued regular adapted process. Moreover, if $(\tau_{n}: n \in \N \cup \{ \infty\})$ are the arrival times of the Poisson process $(N(t, V): t \geq 0)$, the sum in \eqref{eqDefiPoissonInteg} is a fixed random variable on each interval $[\tau_{n},\tau_{n+1})$ and hence the Poisson stochastic integral has c\`{a}dl\`{a}g paths.

For any $\psi \in \Psi$ we define the (weak) Poisson stochastic integral:
\begin{equation} \label{eqDefiWeakPoissonInteg}
\int_{0}^{t} \int_{V} \, R(s,f)' \psi \,   N(ds,df)(\omega) \\ 
= \sum_{0 \leq s \leq t} \inner{\Delta L_{s}(\omega)}{R(s,\omega,\Delta L_{s}(\omega))'\psi} \ind{V}{\Delta L_{s}(\omega)}, 
\end{equation}
which is a real-valued adapted c\`{a}dl\`{a}g process by the reasons explained above. We immediately obtain the \emph{weak-strong compatibility} for Poisson integrals:
\begin{equation}\label{eqWeakStrongCompPoissonInteg}
 \inner{\int_{0}^{t} \int_{V} R(s,f)   N(ds,df)}{\psi}=\int_{0}^{t} \int_{V} \, R(s,f)' \psi \,   N(ds,df). 
\end{equation} 
\end{enumerate}

\begin{remark} \label{remaCompenPoissInteg}
Suppose for $V \in \mathcal{A}$ that $\int_{0}^{T} \int_{V} \abs{\inner{f}{R(r,f)' \phi}}^{2} \, \nu(df)dr < \infty$ for every $T>0$ and $\psi \in \Psi$. We may then define
$$ \int_{0}^{t} \int_{V} \, R(s,f)' \psi \, \widetilde{N}(ds,df)= \int_{0}^{t} \int_{V} \, R(s,f)' \psi \,   N(ds,df)- \int_{0}^{t} \int_{V} \inner{f}{R(r,f)' \phi} \, \nu(df)dr. $$
One can check by going to simple processes (see Definition 4.4 and equation (4.7) in \cite{FonsecaMora:2018-1}) that the above definition is consistent with our earlier definition of the compensated Poisson integral defined with respect to the compensated Poisson random measure \eqref{eqDefiPoissonMartValuMea}.    
\end{remark}

We have now all the necessary ingredients to introduce the stochastic integral with respect to a L\'{e}vy process. 
We start by defining the class of integrands for our theory. 

\begin{definition}\label{defiIntegrandsLevy}
Let $L=(L_{t}: t \geq 0)$ be a $\Phi'$-valued L\'{e}vy process with L\'{e}vy-It\^{o} decomposition  \eqref{eqLevyItoDecomposition}. Let $\Lambda(L,\Phi,\Psi)$ denotes the collection of all the mappings $R: [0, \infty) \times \Omega \times \Phi' \rightarrow \mathcal{L}(\Phi',\Psi')$ satisfying the following properties: 
\begin{enumerate}
\item For every $\psi \in \Psi$, $T>0$,  the mapping $(t,\omega,f) \mapsto R(t,\omega,f)'\psi$ is $\mathcal{P}_{T} \otimes \mathcal{B}(\Phi')/\mathcal{B}(\Phi)$-measurable.
\item For every $T >0$ and $\psi \in \Psi$ the mapping $R$ satisfies 
\begin{equation*}
\Exp \int_{0}^{T} \, \left[  \abs{\inner{\goth{m}}{R(s,0)'\psi}}^{2} + \mathcal{Q}(R(s,0)'\psi)^{2} + \int_{B_{\rho'}(1)} \abs{\inner{f}{R(s,f)'\psi}}^{2} \nu(df)\, \right] ds < \infty.
\end{equation*}
\end{enumerate}
\end{definition}

We define the \emph{(strong) stochastic integral of $R$ with respect to $L$} as the $\Psi'$-valued  regular adapted c\`{a}dl\`{a}g process defined for each $t \geq 0$ by
\begin{eqnarray}
\int_{0}^{t} \int_{\Phi'} R(s,f) L(ds,df)
& = & \int_{0}^{t}  R(s,0) \goth{m}  ds + \int_{0}^{t} \int_{B_{\rho'}(1)}  R(s,f)  M(ds,df)  \nonumber \\
& {} & + \int_{0}^{t} \int_{B_{\rho'}(1)^{c}} R(s,f)   N(ds,df). \label{eqDefiLevyIntegral}   
\end{eqnarray} 

Likewise, for every $t \geq 0$ and $\psi \in \Psi$ we define the \emph{(weak) stochastic integral of $R$ with respect to $L$} as the real-valued adapted process:
\begin{eqnarray}
\int_{0}^{t} \int_{\Phi'} R(s,f)'\psi \, L(ds,df)
 & = & \int_{0}^{t} \inner{\goth{m}}{R(s,0)'\psi} ds
+ \int^{t}_{0} \int_{B_{\rho'}(1)} R (s,f)'\psi M (ds, df) \nonumber \\
& {} & + \int_{0}^{t} \int_{B_{\rho'}(1)^{c}} \, R(s,f)' \psi \,   N(ds,df).  \label{eqDefiWeakLevyIntegral}
\end{eqnarray} 

From  \eqref{eqDefiDetermiIntegDriftGeneral},  \eqref{eqWeakStrongCompLevyMartValMeas} and \eqref{eqWeakStrongCompPoissonInteg} we have the following \emph{weak-strong compatibility}: for every $\psi \in \Psi$, $\Prob$-a.e. $\forall t \in [0,T]$, 
\begin{equation}\label{eqWeakStrongCompLevyInteg}
\inner{\int_{0}^{t} \int_{\Phi'} R(s,f) L(ds,df)}{\psi}
=\int_{0}^{t} \int_{\Phi'} R(s,f)'\psi \, L(ds,df).  
\end{equation} 

\begin{proposition}\label{propActiLineOperLevyInteg} Let $\Upsilon$ be a quasi-complete, bornological, nuclear space and let $S \in \mathcal{L}(\Psi',\Upsilon')$. Then, we have $\Prob$-a.e.
\begin{equation} \label{eqActionLineOpeLevyInteg}
S \left( \int_{0}^{t} \int_{\Phi'}  \,  R(s,f) \, L(ds,df) \right) = \int_{0}^{t} \int_{\Phi'} \, S R(s,f) \, L(ds,df), \quad \forall \, t\geq 0. 
\end{equation}
\end{proposition}
\begin{proof} We only need to verify that \eqref{eqActionLineOpeLevyInteg} holds for each term of \eqref{eqDefiLevyIntegral}. First, from \eqref{eqDefiDetermiIntegDriftGeneral} we have $\Prob$-a.e. $\forall \, t \geq 0$, $\upsilon \in \Upsilon$:
$$\inner{S \left(\int_{0}^{t} R(s,0) \goth{m} ds \right)}{\upsilon}= \int_{0}^{t} \inner{\goth{m}}{R(s,0)' S' \upsilon} ds=\inner{\int_{0}^{t} S R(s,0) \goth{m} ds}{\upsilon}.$$ 
Then, as both $S \left(\int_{0}^{t} R(s,0) \goth{m} ds \right)$ and $\int_{0}^{t} S R(s,0) \goth{m} ds$ are regular processes with continuous paths, the above equality shows that $S \left(\int_{0}^{t} R(s,0) \goth{m} ds \right) = \int_{0}^{t} S R(s,0) \goth{m} ds$ (see Proposition 2.12 in \cite{FonsecaMora:2018}). For the stochastic integral with respect to the L\'{e}vy martingale-valued measure $M$ we have by Proposition 5.18 in \cite{FonsecaMora:2018-1} that 
$$ S \left( \int_{0}^{t} \int_{B_{\rho'}(1)}  R(s,f) M(ds,df) \right) =  \int_{0}^{t} \int_{B_{\rho'}(1)}  S R(s,f) M(ds,df). $$
Finally, by \eqref{eqDefiPoissonInteg} we easily conclude that 
$$ S \left( \int_{0}^{t} \int_{B_{\rho'}(1)^{c}} R(s,f)   N(ds,df) \right) = \int_{0}^{t} \int_{B_{\rho'}(1)^{c}} S R(s,f)   N(ds,df). $$
\end{proof}

\section{Stochastic Evolution Equations With L\'{e}vy Noise}\label{sectExisteUniqueSEE}

\subsection{Evolution Systems}\label{subSectEvoSystem}

For the reader convenience and in order to set our notation, in this section we quickly review some properties of evolution systems and $C_{0}$-semigroups of operators that we will need for our forthcoming arguments.  

A family $(U(s,t): 0 \leq s \leq t < \infty) \subseteq \mathcal{L}(\Psi,\Psi)$ is a \emph{backward evolution system} if $U(t,t)=I$, $U(s,t)=U(s,r)U(r,t)$, $0 \leq s \leq r \leq t$. It is furthermore called \emph{strongly continuous} if for every $\psi \in \Psi$, $s,t \geq 0$, the mappings $[s,\infty) \ni r \mapsto U(s,r)\psi$ and $[0,t] \ni r \mapsto U(r,t)\psi$ are continuous. 

Let $A=(A_{t}:t \geq 0)$ be a family of linear, closed operators on $\Psi$ with domains $\mbox{Dom}(A_{t})$, $t \geq 0$. We say that $A$ \emph{generates} the backward evolution system $(U(s,t): 0 \leq s \leq t < \infty)$ if there are dense subspaces $(D_{t}: t \in \R)$ of $\Psi$ (known as \emph{regularity spaces}), such that $U(s,t) D_{t} \subseteq D_{s} \subseteq \mbox{Dom}(A_{s})$ for $0 \leq s \leq t$ and if the following relations are satisfied: 
\begin{equation}\label{eqForwardEquationDeriv}
\frac{d}{dt} U(s,t)\psi = U(s,t)A(t) \psi, \quad \forall \, \psi \in D_{t}, \, s \leq t.
\end{equation}  
\begin{equation} \label{eqBackwardEquationDeriv}
\frac{d}{ds} U(s,t)\psi = -A(s)U(s,t) \psi, \quad \forall \, \psi \in D_{t}, \, s \leq t.
\end{equation}
Equations \eqref{eqForwardEquationDeriv} and  \eqref{eqBackwardEquationDeriv} are called the \emph{forward and backward equations}. From these equations we can conclude the following useful result (Lemma 1.1 in \cite{KallianpurPerezAbreu:1988}): 
\begin{equation}\label{eqForwardEquation}
U(u,t)\psi = \psi + \int_{u}^{t} \, U(u,s)A(s) \psi \, ds, \quad \forall \, \psi \in D_{t}, \, 0 \leq u \leq t.
\end{equation}  
\begin{equation} \label{eqBackwardEquation}
U(u,t)\psi = \psi + \int_{u}^{t} \, A(s) U(s,t)\psi \, ds, \quad \forall \, \psi \in D_{t}, \, 0 \leq u \leq t.
\end{equation} 
The integrals in \eqref{eqForwardEquation} and \eqref{eqBackwardEquation} are Riemann integrals in $\Psi$.  Further properties of Riemann integrals for functions with real domain and with values in locally convex spaces can be consulted in \cite{AlbaneseBonetRicker:2012, FalbJacobs:1968}.

The backward evolution system $(U(s,t): 0 \leq s \leq t < \infty)$ is called $(C_{0},1)$ if for each $T>0$ and each continuous seminorm $p$ on $\Psi$ there exists some $\vartheta_{p} \geq 0$ and a continuous seminorm $q$ on $\Psi$ such that  $ p(U(s,t)\psi) \leq e^{\vartheta_{p} (t-s)} q(\psi)$, for all $0 \leq s \leq t \leq T$ and $\psi \in \Psi$.

Observe that since $\Psi$ is reflexive, the family of dual operators $(U(t,s)': 0 \leq s \leq t < \infty) \subseteq \mathcal{L}(\Psi',\Psi')$ defines a \emph{forward evolution system},  i.e. $U(t,t)'=I$, $U(t,s)'=U(t,r)'U(r,s)'$, $0 \leq s \leq r \leq t$. Moreover, the dual family is \emph{strongly continuous}, i.e. the mapping $\{(s,t): s \leq t\} \ni (s,t) \mapsto U(s,t)' f$ is continuous for each $f \in \Psi'$. 
Examples of backwards and forward evolution systems that appear on the study of stochastic evolution equations in duals of nuclear spaces can be found in \cite{KallianpurMitoma:1992, KallianpurPerezAbreu:1988, Mitoma:1987} (see also Section \ref{subSectTimeRegu}). 

A family $( S(t): t \geq 0) \subseteq \mathcal{L}(\Psi,\Psi)$  is  a \emph{$C_{0}$-semigroup} on $\Psi$ if \begin{inparaenum}[(i)] \item $S(0)=I$, $S(t)S(s)=S(t+s)$ for all $t, s \geq 0$, and \item $\lim_{t \rightarrow s}S(t) \psi = S(s) \psi$, for all $s \geq 0$ and any $\psi \in \Psi$. \end{inparaenum} The \emph{infinitesimal generator} $A$ of  $( S(t) : t \geq 0)$ is the linear operator on $\Psi$ defined by $ A \psi = \lim_{h \downarrow 0} \frac{S(h) \psi -\psi}{h}$ (limit in $\Psi$), whenever the limit exists; the domain of $A$ being the set $\mbox{Dom}(A) \subseteq \Psi$ for which the above limit exists. Since our assumptions imply that $\Psi$ is reflexive, then the family $( S(t)' : t \geq 0)$ of  dual operators is a $C_{0}$-semigroup on $\Psi'$ with generator $A'$, that we call the \emph{dual semigroup} and the \emph{dual generator} respectively. Moreover, if $( S(t) : t \geq 0)$ is equicontinuous then $( S(t)' : t \geq 0)$ is also equicontinuous (see \cite{Komura:1968}, Theorem 1 and its Corollary).

A $C_{0}$-semigroup $( S(t) : t \geq 0)$ is called a \emph{$(C_{0},1)$-semigroup} if for each continuous seminorm $p$ on $\Psi$  there exists some $\vartheta_{p} \geq 0$ and a continuous seminorm $q$ on $\Psi$ such that  $ p(S(t)\psi) \leq e^{\vartheta_{p} t} q(\psi)$, for all $t \geq 0$ and $\psi \in \Psi$. If in the above inequality $\vartheta_{p}=\omega$ with $\omega$ a positive constant (independent of $p$) $( S(t) : t \geq 0)$ is called \emph{quasiequicontinuous}, and if $\omega=0$ $( S(t) : t \geq 0)$ is called \emph{equicontinuous}. It is worth to mention that even under our assumption that $\Psi$ is reflexive, the dual semigroup $( S(t)' : t \geq 0)$ to a $(C_{0},1)$-semigroup $( S(t) : t \geq 0)$  is not in general a $(C_{0},1)$-semigroup on $\Psi'$ (see  \cite{Babalola:1974}, Section 6). Further properties of $(C_{0},1)$-semigroup and its generator can be consulted in \cite{Babalola:1974}.

Note that if $( S(t) : t \geq 0)$ is a $C_{0}$-semigroup with generator $A$, then $U(s,t)=S(t-s)$ for $0 \leq s \leq t$ defines a backward evolution system with family of generators $A(t)=A$ and $\mbox{Dom}(A_{t})=D_{t}=\mbox{Dom}(A)$ for every $t \geq 0$. If $( S(t) : t \geq 0)$ is a $(C_{0},1)$-semigroup, then $U(s,t)=S(t-s)$ for $0 \leq s \leq t$ is $(C_{0},1)$.

\subsection{Stochastic Convolution For L\'{e}vy Processes}\label{subSectStocConvLevyProc}

The following is the main result of this section which guarantees the existence of the stochastic convolution and some of its properties. 

\begin{theorem}\label{theoExistStochConvLevy}
Let $L=(L(t): t \geq 0)$ be a $\Phi'$-valued L\'{e}vy process with L\'{e}vy-It\^{o} decomposition \eqref{eqLevyItoDecomposition}, $R \in \Lambda(L,\Phi,\Psi)$, and  $(U(s,t): 0 \leq s \leq t < \infty) \subseteq \mathcal{L}(\Psi,\Psi)$ be a $(C_{0},1)$-backward evolution system. The stochastic convolution process 
$$ X_{U,R}(t)\defeq \int_{0}^{t}\int_{\Phi'} \, U(t,s)' R(s,f) \, L(ds,df), \quad \forall \, t \geq 0. $$
is the $\Psi'$-valued regular adapted process defined by
\begin{eqnarray}
\int_{0}^{t}\int_{\Phi'} \, U(t,s)' R(s,f) \, L(ds,df)
& = & \int_{0}^{t}  U(t,s)' R(s,0) \goth{m}  ds  + \int_{0}^{t} \int_{B_{\rho'}(1)}  U(t,s)'R(s,f)  M(ds,df) \nonumber \\ 
& {} & + \int_{0}^{t} \int_{B_{\rho'}(1)^{c}} U(t,s)'R(s,f)  N(ds,df).  \label{eqDefiStochaConvLevy} 
\end{eqnarray}
\end{theorem}
\begin{proof}
We define $ X_{U,R}(0)=0$. Now fix $t>0$ and define $G: [0, \infty) \times \Omega \times \Phi' \rightarrow \mathcal{L}(\Phi',\Psi')$ by $G(s, \omega, f)=\ind{[0,t]}{s} U(t,s)' R(s, \omega, f)$. We must show that $G \in \Lambda(L,\Phi,\Psi)$. 

In effect, Definition \ref{defiIntegrandsLevy}(1) and the strong continuity of the evolution system $(U(s,t): 0 \leq s \leq t < \infty)$ imply that for every $\psi \in \Psi$,   the mapping $(s,\omega,f) \mapsto R(s,\omega,f)'U(s,t)\psi$ defined on $[0,t] \times \Omega \times \Phi'$ is $\mathcal{P}_{t} \otimes \mathcal{B}(\Phi')/\mathcal{B}(\Phi)$-measurable.

Now we must show the integrability condition of Definition \ref{defiIntegrandsLevy}(2). Let $X_{s}(\omega)=R(s,\omega,0) \goth{m}$ for $(s,\omega) \in [0,t] \times \Omega$. The $\Psi'$-valued process $(X_{s}: s \in [0,t])$ is  (weakly) predictable by Definition \ref{defiIntegrandsLevy}(1). Since 
$\Exp \int_{0}^{t} \, \abs{\inner{X_{s}}{\psi}}^{2} \, ds=\Exp \int_{0}^{t} \, \abs{\inner{\goth{m}}{R(s,0)'\psi}}^{2} ds< \infty$ for each $\psi \in \Psi$,  from a mild modification of the arguments used in the proof of Lemma 6.11 in \cite{FonsecaMora:2018-1} we can show that there exists a continuous Hilbertian seminorm $\varrho$ on $\Psi$ such that 
\begin{equation}\label{eqIntegForGothM}
\Exp \int_{0}^{t} \, \varrho'(R(s,0) \goth{m})^{2} \, ds = \Exp \int_{0}^{t} \, \varrho'(X_{s})^{2} \, ds< \infty.
\end{equation} 
Now, since $(U(s,t): 0 \leq s \leq t < \infty)$ is a $(C_{0},1)$-evolution system, there exist $\vartheta_{\varrho} \geq 0$ and a  continuous seminorm $q$ on $\Psi$ such that  $ \varrho(U(s,t)\psi) \leq e^{\vartheta_{\varrho} (t-s)} q(\psi)$, for all $0 \leq s \leq t $, $\psi \in \Psi$. Then it follows from  \eqref{eqIntegForGothM} that
\begin{eqnarray}
\Exp \int_{0}^{t} \, \abs{\inner{\goth{m}}{R(s,0)'U(s,t)\psi}}^{2} \, ds 
& \leq & \Exp \int_{0}^{t} \, \varrho'(R(s,0)\goth{m})^{2}\varrho(U(s,t)\psi)^{2} \, ds \nonumber \\
& \leq & e^{2 \vartheta_{\varrho} t} q(\psi)^{2} \, \Exp \int_{0}^{t} \, \varrho'(R(s,0)\goth{m})^{2} \, ds < \infty. \label{eqFiniMomeDriftStocConv}
\end{eqnarray}

Now, since $R \in \Lambda(L,\Phi,\Psi)$, on the time interval $[0,t]$ we have that $R$ is integrable with respect to the L\'{e}vy martingale-valued measure \eqref{levyMartValuedMeasExam}. Theorem 5.11 in \cite{FonsecaMora:2018-1} shows that there exists a continuous Hilbertian seminorm $p$ on $\Psi$ and $\tilde{R}: [0,t] \times \Omega \times B_{\rho'}(1) \rightarrow \mathcal{L}(\Phi',\Psi'_{p})$ such that $R(s,\omega,f)=i'_{p} \tilde{R}(s,\omega, f)$ for $\mbox{Leb} \otimes \Prob \otimes \nu$-a.e. $(s,\omega,f) \in [0,t]\times \Omega \times B_{\rho'}(1)$ and  
\begin{equation}\label{eqSquaNormIntegForR}
\Exp \int_{0}^{t} \, \norm{\tilde{R}(s,0)}^{2}_{\mathcal{L}_{2}(\Phi'_{\mathcal{Q}},\Psi'_{p})}  ds
+\Exp \int_{0}^{t} \int_{B_{\rho'}(1)} \, p(\tilde{R}(s,f)f)^{2} \nu(df) ds< \infty.
\end{equation}
Let $\vartheta_{p} \geq 0$ and $q$ a  continuous seminorm on $\Psi$ such that  $p(U(s,t)\psi) \leq e^{\vartheta_{p} (t-s)} q(\psi)$ for all $0 \leq s \leq t$ and $\forall \psi \in \Psi$. For every $\psi \in \Psi$ we have by \eqref{eqSquaNormIntegForR} that 
\begin{eqnarray}
\Exp \int_{0}^{t} \,  \mathcal{Q}(R(s,0)' U(s,t)\psi)^{2} \, ds
& \leq & \Exp \int_{0}^{t} \,  \norm{\tilde{R}(s,0)'}^{2}_{\mathcal{L}_{2}(\Psi_{p}, \Phi_{\mathcal{Q}})} p(i_{p} U(s,t)\psi)^{2} \, ds \nonumber \\
& \leq & q(\psi)^{2} e^{2\vartheta_{p} t} \Exp \int_{0}^{t} \,  \norm{\tilde{R}(s,0)}^{2}_{\mathcal{L}_{2}(\Phi'_{\mathcal{Q}},\Psi'_{p})} \, ds < \infty. \label{eqFiniSecoMomeConvWiener}
\end{eqnarray}
Similarly, for every $\psi \in \Psi$ we have by \eqref{eqSquaNormIntegForR} that 
\begin{flalign}
& \Exp \int_{0}^{t} \,  \int_{B_{\rho'}(1)} \abs{\inner{f}{R(s,f)'U(s,t)\psi}}^{2} \nu(df)\,ds \nonumber \\
& \leq  \Exp \int_{0}^{t} \,  \int_{B_{\rho'}(1)} p'(\tilde{R}(s,f)f)^{2} \, p(i_{p}U(s,t)\psi)^{2} \nu(df)\,ds \nonumber \\
& \leq  q(\psi)^{2} e^{2\vartheta_{p} t} \, \Exp \int_{0}^{t} \,  \int_{B_{\rho'}(1)} p'(\tilde{R}(s,f)f)^{2} \nu(df)\,ds < \infty. \label{eqFiniSecoMomeConvCompPoisson}
\end{flalign}
If we collect the estimates \eqref{eqFiniMomeDriftStocConv}, \eqref{eqFiniSecoMomeConvWiener} and \eqref{eqFiniSecoMomeConvCompPoisson} we conclude that $G$ satisfies Definition \ref{defiIntegrandsLevy}(2) and hence  $G \in \Lambda(L,\Phi,\Psi)$. Then we define $X_{U,R}(t) = \int_{0}^{t}\int_{\Phi'} \, G(s,f) \, L(ds,df)$, showing the existence of the stochastic convolution process defined by  \eqref{eqDefiStochaConvLevy}. 
\end{proof}

As a direct consequence of the results in Theorem \ref{theoExistStochConvLevy} and from  \eqref{eqWeakStrongCompLevyInteg}, for any given $\psi \in \Phi$ and $t \geq 0$  we have $\Prob$-a.e.  
\begin{equation}\label{eqWeakStrongCompStochConvLevy}
\inner{\int_{0}^{t}\int_{\Phi'} \, U(t,s)' R(s,f) \, L(ds,df)}{\psi}
=\int_{0}^{t}\int_{\Phi'} \, R(s,f)' U(s,t) \psi \, L(ds,df),  
\end{equation} 
where from \eqref{eqDefiWeakLevyIntegral} it follows that   
\begin{eqnarray}
\int_{0}^{t}\int_{\Phi'} \, R(s,f)' U(s,t) \psi \, L(ds,df)
 & = & \int_{0}^{t} \inner{\goth{m}}{ R(s,f)' U(s,t) \psi} ds  \label{eqDefiWeakStochConvLevy} \\
& {} & + \int^{t}_{0} \int_{B_{\rho'}(1)}  R(s,f)' U(s,t) \psi M (ds, df) \nonumber \\
& {} & + \int_{0}^{t} \int_{B_{\rho'}(1)^{c}} \, \inner{f}{ R(s,f)' U(s,t) \psi} \,   N(ds,df). \nonumber
\end{eqnarray}

\begin{remark}\label{remaDefStochConvC01Semig}
If we have $U(s,t)=S(t-s)$ where $( S(t) : t \geq 0)$ is a $(C_{0},1)$-semigroup, following the standard practice we will use the notation $ \int_{0}^{t}\int_{\Phi'} \, S(t-s)' R(s,f) \, L(ds,df)$ and $\int_{0}^{t}\int_{\Phi'} \, R(s,f)' S(t-s) \psi \, L(ds,df)$
for the stochastic convolution (in the strong and weak sense respectively). Likewise in \eqref{eqDefiStochaConvLevy} we  replace $U(t,s)'$ with $S(t-s)'$ and in \eqref{eqDefiWeakStochConvLevy} we  replace $U(s,t)$ with $S(t-s)$. 
\end{remark}

In the next result we show that the stochastic convolution possesses almost surely finite square moments.  

\begin{proposition}\label{propASFiniMomeStochConv} 
For every $t>0$ and $\psi \in \Psi$ we have $\Prob$-a.e.
\begin{equation}\label{eqStochConvASSecondMoment}
\int_{0}^{t} \, \left( \inner{\int_{0}^{s}\int_{\Phi'} \, U(s,r)' R(r,f) \, L(dr,df)}{\psi} \right)^{2} \, ds  < \infty.  
\end{equation}
Assume moreover that for every $t>0$ and $\psi \in D_{t}$, the mapping $s \mapsto A(s) \psi$ is continuous on $[0,t]$. Then for every $t>0$ and $\psi \in D_{t}$ we have $\Prob$-a.e.
\begin{equation}\label{eqStochConvASSecondMomentFamilyAs}
\int_{0}^{t} \, \left( \inner{\int_{0}^{s}\int_{\Phi'} \, U(s,r)' R(r,f) \, L(dr,df)}{ A(s) \psi} \right)^{2} \, ds  < \infty.  
\end{equation}
\end{proposition}
\begin{proof} Let $t>0$ and $\psi \in \Psi$. 
From \eqref{eqWeakStrongCompStochConvLevy}  it suffices to verify that each term in the right-hand side of \eqref{eqDefiWeakStochConvLevy} has $\Prob$-a.e. a finite second moment with respect to the Lebesgue measure on $[0,t]$.  

From Jensen's inequality, \eqref{eqFiniMomeDriftStocConv}, and Fubini's theorem, it follows that 
\begin{eqnarray*}
\Exp \int_{0}^{t} \left( \int_{0}^{s} \, \inner{\goth{m}}{R(r,0)'U(r,s)\psi} \, dr \right)^{2} ds 
& \leq & \Exp \int_{0}^{t} \left( s \int_{0}^{s} \, \abs{\inner{\goth{m}}{R(r,0)'U(r,s)\psi}}^{2} \, dr \right)  ds \\ 
& \leq & q(\psi)^{2} \Exp \int_{0}^{t} s e^{2 \vartheta_{\varrho} s}  \left( \int_{0}^{s} \, \varrho'(R(r,0)\goth{m})^{2}  \, dr \right) ds \\ 
& \leq &  q(\psi)^{2} \left( \Exp \int_{0}^{t} \, \varrho'(R(r,0)\goth{m})^{2} \, dr \right) \left( \int_{0}^{t} se^{2 \vartheta_{\varrho} s} ds \right).  
\end{eqnarray*} 
Hence we conclude that the first term in the right-hand side of \eqref{eqDefiWeakStochConvLevy} has $\Prob$-a.e. a finite square moment.

Now from  the It\^{o} isometry (Theorem 4.7 in \cite{FonsecaMora:2018-1}), \eqref{eqSquaNormIntegForR},  \eqref{eqFiniSecoMomeConvWiener} and \eqref{eqFiniSecoMomeConvCompPoisson}, we have for every $s \in [0,t]$,
\begin{flalign*}
& \Exp \left(\int^{s}_{0} \int_{B_{\rho'}(1)}  R(r,f)' U(r,s) \psi M (dr, df)  \right)^{2} \\
& =  \Exp \int_{0}^{s} \,  \mathcal{Q}(R(r,0)' U(r,s)\psi)^{2} \, dr
+ \Exp \int_{0}^{s} \,  \int_{B_{\rho'}(1)} \abs{\inner{f}{R(r,f)'U(r,s)\psi}}^{2} \nu(df)\,dr \\
& \leq q(\psi)^{2} e^{2\vartheta_{p} s}  \left( \Exp \int_{0}^{t} \, \norm{\tilde{R}(r,0)}^{2}_{\mathcal{L}_{2}(\Phi'_{\mathcal{Q}},\Psi'_{p})}  dr
+\Exp \int_{0}^{t} \int_{B_{\rho'}(1)} \, p(\tilde{R}(r,f)f)^{2} \nu(df) dr \right).
\end{flalign*}
The above inequality immediately implies that 
$$ \Exp \int_{0}^{t} \left(\int^{s}_{0} \int_{B_{\rho'}(1)}  R(r,f)' U(r,s) \psi M (dr, df)  \right)^{2} ds < \infty. $$
Thus the second term in the right-hand side of \eqref{eqDefiWeakStochConvLevy} has $\Prob$-a.e. a finite second moment.

Finally, since for every $\omega \in \Omega$ the  integral $\int_{0}^{s} \int_{B_{\rho'}(1)^{c}} \, \inner{f}{ R(r,f)' U(r,s) \psi} \,   N(dr,df)$ is defined as a finite sum, it is immediate that we have 
$$\int_{0}^{t} \abs{ \int_{0}^{s} \int_{B_{\rho'}(1)^{c}} \, \inner{f}{ R(r,f)' U(r,s) \psi}^{2} \,   N(dr,df) } ds < \infty. $$
Therefore the third term in the right-hand side of \eqref{eqDefiWeakStochConvLevy} has $\Prob$-a.e. a finite second moment. We have shown that  \eqref{eqStochConvASSecondMoment} holds. 

The proof of \eqref{eqStochConvASSecondMomentFamilyAs} can be carried out from similar arguments to those used above for \eqref{eqStochConvASSecondMoment} replacing $\psi$ by $A(s)\psi$ and using the fact that the hypothesis of continuity on $[0,t]$ of the mapping $s \mapsto A(s) \psi$ implies that $\sup_{0 \leq s \leq t} q(A(s)\psi)<\infty$  for every continuous seminorm $q$ on $\Psi$.  
\end{proof}

\subsection{Existence and Uniqueness of Solutions}


In this section we show the existence and uniqueness of solutions to the following class of (time-inhomogeneous) stochastic evolution equations, 
\begin{equation}\label{eqSEELevy}
d Y_{t}=A(t)' Y_{t}dt +\int_{\Phi'} R(t,f) L(dt,df), \quad t \geq 0, 
\end{equation}
with initial condition $Y_{0}=\eta$ $\Prob$-a.e., where we will assume the following:

\begin{assumption}\label{assuSEELevy}
\begin{enumerate}[label=(A\arabic*)]
\item $\eta$ is a $\mathcal{F}_{0}$-measurable $\Phi'$-valued regular random variable.
\item $(U(s,t): 0 \leq s \leq t < \infty) \subseteq \mathcal{L}(\Psi,\Psi)$ is a $(C_{0},1)$ backward evolution system with family of generators $(A(t): t \geq 0)$ with corresponding regularity spaces $(D_{t}: t \geq 0)$. We will make the following additional assumptions: 
\begin{enumerate}
\item For every $t>0$ and $\psi \in D_{t}$, the mapping $s \mapsto A(s) \psi$ is continuous on $[0,t]$. 
\item For every $0 \leq s < t$ and  $\psi \in D_{t}$ we have $A(s) U(s,t) \psi \in D_{s}$. 
\end{enumerate}
\item $L=(L(t): t \geq 0)$ is a $\Phi'$-valued L\'{e}vy process with L\'{e}vy-It\^{o} decomposition \eqref{eqLevyItoDecomposition}.
\item $R \in \Lambda(L,\Phi,\Psi)$. 
\end{enumerate}
\end{assumption}

\begin{remark} The conditions (a),(b) in Assumption \ref{assuSEELevy}(A2) are fulfilled in the following two important cases: 
\begin{enumerate}
\item Suppose that $U(s,t)=S(t-s)$, $0 \leq s \leq t$, for a $(C_{0},1)$-semigroup $(S(t):t\geq 0)$. Since $A(t)=A$ and $\mbox{Dom}(A_{t})=D_{t}=\mbox{Dom}(A)$ for every $t \geq 0$ we immediately get Assumption \ref{assuSEELevy}(A2)(a). Moreover, basic properties of $C_{0}$-semigroups 
(see e.g.  \cite{Komura:1968}) shows that $A(s)U(s,t)\psi=AS(t-s)\psi=S(t-s)A\psi \in \mbox{Dom}(A)$ for every $\psi \in \Psi$, which shows Assumption \ref{assuSEELevy}(A2)(b).
\item Suppose that $A(t) \in \mathcal{L}(\Psi,\Psi)$ for each $t \geq 0$ and the mapping $t \mapsto A(t) \psi$ is continuous from $[0,\infty)$ into $\Psi$ for every $\psi \in \Psi$. The latter assumption is stronger than Assumption \ref{assuSEELevy}(A2)(a). From the former assumption we have $A(s) U(s,t) \psi \in D_{s}=\mbox{Dom}(A(s))=\Psi$ for every $0 \leq s < t$ and  $\psi \in \Psi$. So we obtain Assumption \ref{assuSEELevy}(A2)(b).
\end{enumerate}
\end{remark}

In this work we are interested in the existence and uniqueness of weak solutions. 

\begin{definition}
A $\Psi'$-valued regular adapted process $Y=(Y_{t}: t \geq 0)$ is called a \emph{weak solution} to \eqref{eqSEELevy} if for any given $t \geq 0$, for each $\psi \in D_{t}$ we have $\int_{0}^{t} \abs{\inner{Y_{s}}{A(s)\psi}} ds< \infty$ $\Prob$-a.e. and  
\begin{equation}\label{eqDefWeakSoluSEELevy}
\inner{Y_{t}}{\psi}=\inner{\eta}{\psi} +\int_{0}^{t} \inner{Y_{s}}{A(s)\psi} ds+\int_{0}^{t} \int_{\Phi'} \, R(s,f)' \psi \,L(ds,df). 
\end{equation}
\end{definition}

We start by considering the existence of a weak solution to \eqref{eqSEELevy}. 

\begin{theorem}\label{theoExistSoluSEELevy} The stochastic evolution equation \eqref{eqSEELevy} has a weak solution given by the mild (or evolution) solution 
\begin{equation}\label{eqDefMildSoluSEELevy}
X_{t}=U(t,0)'\eta +\int_{0}^{t}\int_{\Phi'} \, U(t,s)' R(s,f) \, L(ds,df), \quad \forall \, t \geq 0. 
\end{equation}
\end{theorem}

The key step in the proof of Theorem \ref{theoExistSoluSEELevy} is the following result on iterated integration for the stochastic convolution. 

\begin{lemma}\label{lemmFubiniStochConvLevy} Given $t>0$, for every $\psi \in D_{t}$, we have $\Prob$-a.e.
\begin{flalign} 
& \int_{0}^{t} \left( \int_{0}^{s}\int_{\Phi'} \,  R(r,f)' U(r,s) A(s) \psi \, L(dr,df) \right)ds \nonumber \\ 
& = \int_{0}^{t} \int_{\Phi'} \,  R(r,f)'U(r,t) \psi \, L(dr,df) -\int_{0}^{t} \int_{\Phi'} \, R(r,f)' \psi  \,L(dr,df) \nonumber \\
& = \int_{0}^{t} \left( \int_{0}^{s}\int_{\Phi'} \,  R(r,f)' A(s) U(s,t)  \psi \, L(dr,df) \right)ds. \label{eqFubiniStochConvoLevyWeak} 
\end{flalign}
\end{lemma}
\begin{proof}
Let $t>0$ and consider any $\psi \in D_{t}$. We start by showing the equality of the first two lines in \eqref{eqFubiniStochConvoLevyWeak}. It suffices to check that the equality holds for each term in \eqref{eqDefiWeakStochConvLevy}. 

For the first term in \eqref{eqDefiWeakStochConvLevy}, observe that from Fubini's theorem and the forward equation \eqref{eqForwardEquation}, $\omega$-wise we have that
\begin{flalign*}
& \int_{0}^{t} \left( \int_{0}^{s} \inner{\goth{m}}{R(r,0)'U(r,s)A(s) \psi}  \, dr \right)ds  \\
& =  \int_{0}^{t} \left( \int_{r}^{t} \inner{R(r,0) \goth{m}}{U(r,s)A(s) \psi}  \, ds \right)dr \\  
& =  \int_{0}^{t} \inner{R(r,0) \goth{m}}{U(r,t)\psi -\psi}   \, dr \\  
& =  \int_{0}^{t} \inner{\goth{m}}{R(r,0)'U(r,t)\psi} \, dr -\int_{0}^{t} \inner{\goth{m}}{R(r,0)'\psi} \, dr.  
\end{flalign*}   

The second term in \eqref{eqDefiWeakStochConvLevy}, corresponds to the (weak) stochastic integral with respect to the L\'{e}vy martingale-valued measure $M$ of  \eqref{levyMartValuedMeasExam}. From the stochastic Fubini theorem (Theorem 4.24 in \cite{FonsecaMora:2018-1}) and the forward equation \eqref{eqForwardEquation} it follows that $\Prob$-a.e.
\begin{flalign*}
& \int_{0}^{t} \left( \int^{s}_{0} \int_{B_{\rho'}(1)} R(r,f)' U(r,s) A(s) \psi M (dr, df) \right) ds \\
& = \int_{0}^{t} \int_{B_{\rho'}(1)} \left( \int^{t}_{r}  R(r,f)' U(r,s) A(s) \psi \, ds  \right) M (dr, df) \\
& = \int_{0}^{t} \int_{B_{\rho'}(1)}  R(r,f)' U(r,t) \psi - R(r,f)' \psi \, M (dr, df) \\
& =  \int^{t}_{0} \int_{B_{\rho'}(1)} R(r,f)' U(r,t) \psi M (dr, df) - \int^{t}_{0} \int_{B_{\rho'}(1)} R(r,f)' \psi  M (dr, df). 
\end{flalign*}   

For the third term in \eqref{eqDefiWeakStochConvLevy}, from \eqref{eqDefiWeakPoissonInteg}, Fubini's theorem and the forward equation \eqref{eqForwardEquation}, $\omega$-wise we have that 
\begin{flalign*}
& \int_{0}^{t} \left( \int_{0}^{s} \int_{B_{\rho'}(1)^{c}} \, \inner{f}{R(r,f)' U(r,s) A(s) \psi} \,   N(dr,df) \right) ds \\
& =  \int_{0}^{t} \left( \sum_{0 \leq r \leq s} \inner{\Delta L_{r}}{R(r,\Delta L_{r})' U(r,s)A(s) \psi}  \ind{B_{\rho'}(1)^{c}}{\Delta L_{r}} \right)ds \\
& =  \sum_{0 \leq r \leq t} \left( \int_{r}^{t} \inner{R(r,\Delta L_{r}) \Delta L_{r}}{U(r,s)A(s) \psi} ds \right) \ind{B_{\rho'}(1)^{c}}{\Delta L_{r}}  \\  
& =  \sum_{0 \leq r \leq t} \left( \inner{R(r,\Delta L_{r}) \Delta L_{r}}{U(r,t)\psi- \psi} \right) \ind{B_{\rho'}(1)^{c}}{\Delta L_{r}}  \\  
& =  \int_{0}^{t} \int_{B_{\rho'}(1)^{c}} \inner{f}{R(r,f)' U(r,t) \psi}  N(dr,df) -\int_{0}^{t} \int_{B_{\rho'}(1)^{c}} \inner{f}{R(r,f)'\psi}  N(dr,df).  
\end{flalign*}   
The equality of the second and third lines  in \eqref{eqFubiniStochConvoLevyWeak} can be proved   following completely analogous arguments to those used in the above paragraphs but this time using the backward  equation \eqref{eqBackwardEquation}. Details are leaved to the reader. 
\end{proof}

\begin{corollary}\label{coroEqFubiniStochConvoLevyStrong}
Given $t>0$, for every $\psi \in D_{t}$, we have $\Prob$-a.e.
\begin{multline} \label{eqFubiniStochConvoLevyStrong}
\int_{0}^{t} \left( \inner{\int_{0}^{s}\int_{\Phi'} \, U(t,r)' R(r,f) \, L(dr,df)}{A(s) \psi} \right) ds \\
 = \inner{\int_{0}^{t}\int_{\Phi'} \, U(t,r)' R(r,f) \, L(dr,df)}{\psi} -\inner{\int_{0}^{t} \int_{\Phi'} \, R(r,f) \,L(dr,df)}{\psi}, 
\end{multline}
\end{corollary}
\begin{proof}
The equality in \eqref{eqFubiniStochConvoLevyStrong} follows from \eqref{eqFubiniStochConvoLevyWeak}, by  considering the compatibility between weak and strong integrals in \eqref{eqWeakStrongCompLevyInteg} and \eqref{eqWeakStrongCompStochConvLevy}. 
\end{proof}

\begin{proof}[Proof of Theorem \ref{theoExistSoluSEELevy}] 
Let $(X_{t}: t \geq 0)$ be as defined in \eqref{eqDefMildSoluSEELevy}. 
We already know from Theorem \ref{theoExistStochConvLevy} that the stochastic convolution is a $\Psi'$-valued regular adapted process. The strong continuity of the forward evolution system $(U(t,s)': 0 \leq s \leq t)$ and our assumptions on $\eta$ imply that $(U(t,0)'\eta: t \geq 0)$ is a $\Psi'$-valued regular adapted process, so is $(X_{t}: t \geq 0)$. 

Let $t \geq 0$ and $\psi \in D_{t}$. We must show that $\int_{0}^{t} \abs{\inner{X_{s}}{A(s)\psi}} ds< \infty$ $\Prob$-a.e. In effect, from our assumptions the mapping $s \mapsto \inner{U(s,0)'\eta}{A(s)\psi}$ has continuous paths on $[0,t]$. Hence, we have $\int_{0}^{t} \abs{\inner{U(s,0)'\eta}{A(s)\psi}} ds< \infty$ $\Prob$-a.e. Then from \eqref{eqStochConvASSecondMomentFamilyAs} and \eqref{eqDefMildSoluSEELevy} we conclude that $\int_{0}^{t} \abs{\inner{X_{s}}{A(s)\psi}} ds< \infty$ $\Prob$-a.e. 

Our next objective is to show that $(X_{t}: t \geq 0)$ satisfies \eqref{eqDefWeakSoluSEELevy}. 
Fix $t \geq 0$ and choose any $\psi \in D_{t}$. From \eqref{eqDefMildSoluSEELevy} for each $s \in [0,t]$ we have
$$ \inner{X_{s}}{A(s)\psi}=\inner{U(s,0)'\eta}{A(s)\psi}
+\inner{ \int_{0}^{s}\int_{\Phi'} \, U(s,r)' R(r,f) \, L(dr,df)}{A(s)\psi}. $$
Integrating both sides on $[0,t]$ with respect to the Lebesgue measure, then from \eqref{eqFubiniStochConvoLevyStrong} and the forward equation \eqref{eqForwardEquation}, we have $\Prob$-a.e.
\begin{flalign*}
& \int_{0}^{t} \inner{X_{s}}{A(s) \psi} ds \\
& =  \int_{0}^{t} \inner{\eta}{U(0,s) A(s)\psi} ds 
+ \int_{0}^{t} \inner{ \int_{0}^{s} \int_{\Phi'} \, U(s,r)' R(r,f) \, L(dr,df)}{A(s)\psi} ds  \\
& =  \inner{\eta}{U(0,t)\psi-\psi} + \inner{\int_{0}^{t} \int_{\Phi'} \, U(t,r)' R(r,f) \, L(dr,df)}{\psi}-\inner{\int_{0}^{t} \int_{\Phi'} \, R(r,f) \, L(dr,df)}{\psi} \\
& =  \inner{X_{t}}{\psi} -\inner{\eta}{\psi}- \inner{\int_{0}^{t} \int_{\Phi'} \, R(r,f) \, L(dr,df)}{\psi}.  
\end{flalign*} 
Hence, the process $X=(X_{t}: t \geq 0)$ defined by \eqref{eqDefMildSoluSEELevy} is a weak solution to 
\eqref{eqSEELevy}. 
\end{proof}

The mild solution possesses almost surely square integrable paths as the next result shows. 

\begin{proposition}\label{propASFiniteMomentMildSol}
Let $X=(X_{t}: t \geq 0)$ be the mild solution  \eqref{eqDefMildSoluSEELevy} to \eqref{eqSEELevy}.
Then, $X$ has almost surely finite square moments, i.e. $\forall t >0$, $\psi \in \Psi$, we have $\Prob$-a.e. $\int_{0}^{t} \abs{\inner{X_{s}}{\psi}}^{2} ds< \infty$. 
\end{proposition}
\begin{proof}
From Proposition \ref{propASFiniMomeStochConv} the stochastic convolution has almost surely finite square  moments. The same holds true for the process $(U(t,0)'\eta: t \geq 0)$ as a direct consequence of the strong continuity of the forward evolution system $(U(t,s)': 0 \leq s \leq t< \infty)$. From \eqref{eqDefMildSoluSEELevy} we have the corresponding result for $X$.  
\end{proof}

Our next main issue in this section concerns the uniqueness of solutions to \eqref{eqSEELevy}. We will require some extra regularity on the trajectories of the solution which is explained in the following definition. 

\begin{definition}\label{defiASLocalBocnTraje}
We say that a $\Psi'$-valued adapted process $Y=(Y_{t}: t \geq 0)$ has \emph{almost surely locally Bochner integrable trajectories} if for each $t>0$ there exists $\Omega_{t} \subseteq \Omega$ with $\Prob(\Omega_{t})=1$ and such that for each $\omega \in \Omega_{t}$ there exists a continuous seminorm $p=p(t,\omega)$ on $\Psi$ such that $Y_{s}(\omega) \in \Psi_{p}$ a.e. on $[0,t]$ and $\int_{0}^{t} \, p'(Y_{s}(\omega)) \, ds< \infty$. 
\end{definition}

The property of almost surely locally Bochner integrable trajectories is implied by other stronger regularity properties of the paths as for example are the existence of finite moments (see Theorem \ref{theoExisUniqStronSquaIntegSolution}) and for processes having c\`{a}dl\`{a}g trajectories as the next result shows. 

\begin{proposition}\label{propCadlagImplyASBochner}
Let $Y=(Y_{t}: t \geq 0)$ be a $\Psi'$-valued regular adapted c\`{a}dl\`{a}g process. Then $Y$ has almost surely locally Bochner integrable trajectories.  
\end{proposition}
\begin{proof}
Since $Y$ is c\`{a}dl\`{a}g there exists $\Omega_{0} \subseteq \Omega$ such that $\Prob(\Omega_{0})=1$ and for each $\omega \in \Omega$ we have that $s \mapsto X_{s}(\omega) \in D_{\infty}(\Psi')$. As $\Psi$ is barrelled,  given any $t>0$ and $\omega \in \Omega_{0}$ there exists a continuous seminorm $p$ on $\Psi$ such that $s \mapsto Y_{s}(\omega) \in D_{t}(\Psi'_{p})$ (Remark 3.6 in \cite{FonsecaMora:Skorokhod}). Hence, we have
$$ \int_{0}^{t} \, p'(Y_{s}(\omega)) \, ds \leq t \, \sup_{0 \leq s \leq t}  p'(Y_{s}(\omega)) < \infty. $$
Hence $Y$ has almost surely locally Bochner integrable trajectories.
\end{proof}

We are ready for our main result on the uniqueness of solutions to \eqref{eqSEELevy}.

\begin{theorem}\label{theoUniqueSoluSEELevy} 
Let $Y=(Y_{t}: t \geq 0)$ be a weak solution  to \eqref{eqSEELevy} which has almost surely locally Bochner integrable trajectories.  
Then for each $t \geq 0$, $Y_{t}$ is given $\Prob$-a.e. by 
$$ Y_{t}=U(t,0)'\eta+\int_{0}^{t} \int_{\Phi'} \,  U(t,r)' R(r,f)\, L(dr,df).$$
\end{theorem}
\begin{proof}
Let $t \geq 0 $ and $\psi \in D_{t}$. We start by showing  that $\Prob$-a.e. 
\begin{equation}\label{eqCondFubiUniqSolSEE}
\int_{0}^{t} \left( \int_{0}^{t} \, \abs{  \mathbbm{1}_{[0,s]}(r) \inner{Y_{r}}{A(r)A(s)U(s,t) \psi} } dr \right) ds < \infty. 
\end{equation}
Let $\omega \in \Omega_{t}$ and $p=p(t,\omega)$ as in Definition \ref{defiASLocalBocnTraje}. From Assumption \ref{assuSEELevy}(A2) it is clear that the set $\{ A(r)A(s)U(s,t) \psi: 0 \leq r \leq s \leq t\}$ is bounded in $\Psi$, hence bounded under $p$. Then we have 
\begin{multline*}
\int_{0}^{t} \left( \int_{0}^{t} \, \abs{  \mathbbm{1}_{[0,s]}(r) \inner{Y_{r}(\omega)}{A(r)A(s)U(s,t) \psi} } dr \right) ds \\
 \leq  t \sup_{0 \leq r \leq s \leq t} p(A(r)A(s)U(s,t) \psi) \, \int_{0}^{t} \, p'(Y_{s}(\omega)) \, ds< \infty.
\end{multline*}

Now from an application of Fubini's theorem (which we can apply thanks to \eqref{eqCondFubiUniqSolSEE}), then using that $A(r)$ is a closed operator for every $r \in [0,t]$, properties of the Riemann integral, and then using the backward equation \eqref{eqBackwardEquation} we have for $\omega \in \Omega_{t}$, 
\begin{eqnarray}
\int_{0}^{t} \left( \int_{0}^{s} \, \inner{Y_{r}(\omega)}{A(r)A(s)U(s,t) \psi} dr \right) ds 
& = & \int_{0}^{t} \left( \int_{r}^{t} \, \inner{Y_{r}(\omega)}{A(r)A(s)U(s,t) \psi} ds \right) dr  \nonumber \\
& = & \int_{0}^{t}  \inner{Y_{r}(\omega)}{A(r) \int_{r}^{t} A(s)U(s,t) \psi ds}  \,  dr  \nonumber \\
& = & \int_{0}^{t}  \inner{Y_{r}(\omega)}{A(r)(U(r,t) \psi-\psi)}  dr  \label{eqFubiniGeneratorEvoluSystem}
\end{eqnarray}
Given any $s \in [0,t]$, from the definition of weak solution (with $\psi$ replaced by $A(s)U(s,t) \psi$), we have $\Prob$-a.e.
\begin{multline*}
\int_{0}^{s}\int_{\Phi'}  \, R(r,f)'A(s)U(s,t) \psi \, L(dr,df) \\
 =  \inner{Y_{s}-\eta}{A(s)U(s,t) \psi}- \int_{0}^{s} \, \inner{Y_{r}}{A(r)A(s)U(s,t)\psi}\, dr.  
\end{multline*}
Integration both sides on $[0,t]$ with respect to the Lebesgue measure and then using \eqref{eqFubiniGeneratorEvoluSystem} and the backward equation \eqref{eqBackwardEquation} it follows that
\begin{flalign*}
& \int_{0}^{t} \left( \int_{0}^{s}\int_{\Phi'}  \, R(r,f)'A(s)U(s,t) \psi \, L(dr,df) \right) ds  \\
& =  \int_{0}^{t} \inner{Y_{s}}{A(s)U(s,t) \psi} ds-\int_{0}^{t} \inner{\eta}{A(s)U(s,t) \psi} ds \\
& - \int_{0}^{t} \left( \int_{0}^{s} \, \inner{Y_{r}}{A(r)A(s)U(s,t)\psi}\, dr \right) ds \\ 
& =   \int_{0}^{t}  \inner{Y_{s}}{A(s)\psi}  ds -\left( \inner{\eta}{U(0,t)\psi}- \inner{\eta}{\psi} \right).  
\end{flalign*}
Reordering, and then from \eqref{eqFubiniStochConvoLevyWeak} we get that 
\begin{flalign*}
& \inner{\eta}{\psi} + \int_{0}^{t} \inner{Y_{s}}{A(s) \psi} ds \\
& = \inner{U(t,0)'\eta}{\psi}+ \int_{0}^{t} \left( \int_{0}^{s} \, R(r,f)'A(s)U(s,t) \psi \, L(dr,df) \right) ds \\
& =\inner{U(t,0)'\eta}{\psi} +\int_{0}^{t} \int_{\Phi'} \,  R(r,f)'U(r,t) \psi \, L(dr,df) -\int_{0}^{t} \int_{\Phi'} \, R(r,f)' \psi  \,L(dr,df). 
\end{flalign*} 
Summing the term $\int_{0}^{t} \int_{\Phi'} \, R(r,f)' \psi  \,L(dr,df)$ at both sides, then using  \eqref{eqWeakStrongCompStochConvLevy} and the fact that $Y$ is a weak solution we conclude that $\Prob$-a.e
\begin{equation}\label{eqWeakSolIsWeaklyMildSol}
\inner{Y_{t}}{\psi}=\inner{U(t,0)'\eta}{\psi} + \inner{\int_{0}^{t} \int_{\Phi'} \,  U(t,r)' R(r,f)\, L(dr,df)}{\psi}.
\end{equation}
Since \eqref{eqWeakSolIsWeaklyMildSol} is valid for each $\psi$ in the dense subset $D_{t}$ of $\Psi$, we conclude from \eqref{eqWeakStrongCompStochConvLevy} that 
$$ Y_{t}=U(t,0)'\eta+\int_{0}^{t} \int_{\Phi'} \,  U(t,r)' R(r,f)\, L(dr,df).$$ 
\end{proof}

We can combine the results of Theorems \ref{theoExistSoluSEELevy} and \ref{theoUniqueSoluSEELevy} to obtain the following result on existence and uniqueness of solutions to stochastic evolution equations in the time-homogeneous setting, i.e. when $A(t)=A$. 

\begin{corollary} Suppose that $A$ is the generator of a $(C_{0},1)$-semigroup $(S(t):t\geq 0)$ on $\Psi$. 
The stochastic evolution equation
\begin{equation}\label{eqSEETimeHomoLevy}
d Y_{t}=AY_{t}dt +\int_{\Phi'} R(t,f) L(dt,df), \quad t \geq 0, 
\end{equation}
with initial condition $Y_{0}=\eta$ $\Prob$-a.e.  has a weak solution given by the mild solution 
\begin{equation*}
X_{t}=S(t)'\eta +\int_{0}^{t}\int_{\Phi'} \, S(t-s)' R(s,f) \, L(ds,df), \quad \forall \, t \geq 0. 
\end{equation*}
Moreover, any weak solution $Y=(Y_{t}: t \geq 0)$ to \eqref{eqSEETimeHomoLevy} which has almost surely locally Bochner integrable trajectories is a version of $X=(X_{t}: t \geq 0)$. 
\end{corollary}

\section{Properties of the Solution}\label{sectPropertiesSolution}

In this section we study different properties of the solution to \eqref{eqSEELevy}. Apart from Assumption \ref{assuSEELevy}, some other assumptions might be considered in each subsection.  

\subsection{Square Integrable Solutions}

In this section we consider the existence of solutions with (strong) square moments. For this purpose, we naturally require the L\'{e}vy process $L=(L_{t}: t \geq 0)$ to be square integrable. In such a case by the L\'{e}vy-It\^{o} decomposition \eqref{eqLevyItoDecomposition} we must have that the Poisson integral process $ \int_{B_{\rho'}(1)^{c}} f N (t,df)$, $t \geq 0$, 
corresponding to the large jumps of $L$ is square integrable. Thus we have $\int_{B_{\rho'}(1)^{c}} \abs{\inner{f}{\phi}}^{2} \, \nu(df) < \infty$, $\forall \phi \in \Phi$. In particular, it follows that 
$\goth{n} \in \Phi'$ where $\inner{\goth{n}}{\phi}=\int_{B_{\rho'}(1)^{c}} \inner{f}{\phi} \, \nu(df)$ for every $\phi \in \Phi$. Moreover, the $\Phi'$-valued compensated Poisson integral $ \int_{B_{\rho'}(1)^{c}} f \widetilde{N} (t,df)= \int_{B_{\rho'}(1)^{c}} f N (t,df) - t\goth{n}$, $t \geq 0$, 
is a zero-mean, square integrable c\`{a}dl\`{a}g L\'{e}vy process and $\Exp \left( \abs{\inner{\int_{B_{\rho'}(1)^{c}} f \widetilde{N} (t,df)}{\phi}}^{2} \right)= t \int_{B_{\rho'}(1)^{c}} \abs{\inner{f}{\phi}}^{2} \, \nu(df)$, $ \forall \phi \in \Phi$, $t \geq 0$ (see Section 4.1 in \cite{FonsecaMora:Levy}). 

From the above arguments and \eqref{eqLevyItoDecomposition} we can decompose $L$ as
\begin{equation} \label{eqLevyItoDecompoSquareIntegrable}
L_{t}=t\widetilde{\goth{m}}+W_{t}+\int_{\Phi'} f \widetilde{N} (t,df),
\end{equation}
for $\widetilde{\goth{m}}=\goth{m}+\goth{n} \in \Phi'$ and 
where
$$ \int_{\Phi'} f \widetilde{N} (t,df) \defeq  \int_{B_{\rho'}(1)} f \widetilde{N} (t,df)+ \int_{B_{\rho'}(1)^{c}} f \widetilde{N} (t,df).$$
 
The decomposition \eqref{eqLevyItoDecompoSquareIntegrable} of $L$ motivates the introduction of the following class of integrands.  
 
\begin{definition}
Let $\Lambda^{2}(L,\Phi,\Psi)$ denotes the  collection of all $R \in \Lambda(L,\Phi,\Psi)$ satisfying for every $T >0$ and $\psi \in \Psi$ that
\begin{equation*}
\Exp \int_{0}^{T} \, \left[  \abs{\inner{\goth{m}}{R(s,0)'\psi}}^{2} + \mathcal{Q}(R(s,0)'\psi)^{2} + \int_{\Phi'} \abs{\inner{f}{R(s,f)'\psi}}^{2} \nu(df)\, \right] ds < \infty.
\end{equation*}
\end{definition}

Assume $R \in \Lambda^{2}(L,\Phi,\Psi)$. 
By Remark \ref{remaCompenPoissInteg} (for $V=B_{\rho'}(1)^{c}$) and \eqref{eqDefiWeakStochConvLevy}, for each $t \geq 0$ and $\psi \in \Psi$ we have $\Prob$-a.e.
\begin{eqnarray}
\int_{0}^{t}\int_{\Phi'} \, R(s,f)' U(s,t) \psi \, L(ds,df)
 & = & \int_{0}^{t} \inner{\goth{m}}{ R(s,f)' U(s,t) \psi} ds \nonumber \\
& {} &  + \int_{0}^{t} \int_{B_{\rho'}(1)^{c}} \, \inner{f}{R(s,f)'  U(s,t) \psi} \nu(df) \, ds \nonumber \\
& {} & + \int^{t}_{0} \int_{\Phi'}  R(s,f)' U(s,t) \psi M (ds, df), \label{eqStochConvLevySquareInteg}
\end{eqnarray} 
where the stochastic integral with respect to the L\'{e}vy martingale-valued measure $M$ on $\Phi'$ is well-defined since by our square integrability assumption on $L$ we have $\int_{\Phi'} \abs{\inner{f}{\phi}}^{2} \, \nu(df) < \infty$, $\forall \phi \in \Phi$. In \eqref{eqStochConvLevySquareInteg} we have used that $\Prob$-a.e.
\begin{eqnarray*}
\int^{t}_{0} \int_{\Phi'}  R(s,f)' U(s,t) \psi M (ds, df) & = & \int^{t}_{0} \int_{B_{\rho'}(1)}  R(s,f)' U(s,t) \psi M (ds, df) \\
& {} & + \int^{t}_{0} \int_{B_{\rho'}(1)^{c}}  R(s,f)' U(s,t) \psi M (ds, df), 
\end{eqnarray*}
which can be justified using Proposition 4.12 in \cite{FonsecaMora:2018-1}. We will apply \eqref{eqStochConvLevySquareInteg} to obtain the existence of strong seconds moments for the stochastic convolution:

\begin{proposition}\label{propStochConvStrongSquareMoment}
Suppose that $L$ is square integrable and $R \in \Lambda^{2}(L,\Phi,\Psi)$. 
For every $t>0$, there exists a continuous Hilbertian seminorm $\varrho$ on $\Psi$ such that 
\begin{equation}\label{eqStochConvStrongSquareMoment}
 \Exp \int_{0}^{t} \, \varrho' \left( \int_{0}^{s} \int_{\Phi'} \, U(s,r)' R(r,f) \, L(dr,df) \right)^{2} ds < \infty. 
\end{equation} 
\end{proposition}
\begin{proof} Let $t>0$. Our first task will be to demonstrate that 
\begin{equation}\label{eqFiniSecMomeStocConvLevySquaInte}
\Exp \int_{0}^{t} \left(\int^{s}_{0} \int_{\Phi'} \,  R(r,f)' U(r,s) \psi \, L(dr, df)  \right)^{2} ds < \infty, \quad \psi \in \Psi. 
\end{equation}
To prove \eqref{eqFiniSecMomeStocConvLevySquaInte}, observe that by the arguments used in the proof of Proposition \ref{propASFiniMomeStochConv} we have 
$$ \Exp \int_{0}^{t} \left( \int_{0}^{s} \, \inner{\goth{m}}{R(r,0)'U(r,s)\psi} \, dr \right)^{2} ds < \infty, $$
$$ \Exp \int_{0}^{t} \left(\int^{s}_{0} \int_{\Phi'}  R(r,f)' U(r,s) \psi M (dr, df)  \right)^{2} ds < \infty. $$
Hence in view of \eqref{eqStochConvLevySquareInteg} to prove \eqref{eqFiniSecMomeStocConvLevySquaInte} it is sufficient to check that 
$$ \Exp \int_{0}^{t} \left(  \int_{0}^{s} \,  \int_{B_{\rho'}(1)^{c}} \inner{f}{R(r,f)' U(r,s) \psi} \nu(df)\,  dr \right)^{2} ds < \infty. $$
To do this, note that our assumption that $R \in \Lambda^{2}(L,\Phi,\Psi)$ implies that $r: \Psi \rightarrow [0,\infty)$ given by
$$ r(\psi)= \left( \Exp \int_{0}^{t} \,  \int_{B_{\rho'}(1)^{c}} \abs{\inner{f}{R(r,f)'\psi}}^{2} \nu(df)\,  dr \right)^{1/2}, \quad \forall \psi \in \Psi, $$
is a well-defined Hilbertian seminorm on $\Psi$. Moreover, 
since the mapping $\psi \mapsto \abs{\inner{f}{R(r,f)'\psi}}^{2}$ is continuous on $\Psi$ for each $(r,\omega,f) \in [0,t] \times \Omega \times B_{\rho'}(1)^{c}$, from an application of Fatou's Lemma we can check that $r$ is lower-semicontinuous. Then Proposition 5.7 in \cite{FonsecaMora:Skorokhod} shows that $r$ is continuous on $\Psi$. As $\Psi$ is a nuclear space, there exists a continuous Hilbertian seminorm $q$ on $\Psi$ such that $r \leq q$ and $i_{r,q}$ is Hilbert-Schmidt. Then if $(\psi^{q}_{j}: j \in \N) \subseteq \Psi$ is a complete orthonormal system in $\Psi_{q}$, we have
\begin{eqnarray*}
\Exp \int_{0}^{t} \int_{B_{\rho'}(1)^{c}} q' \left(  R(r,f)f \right)^{2} \nu(df)\,  dr  
& = & \Exp \int_{0}^{t} \int_{B_{\rho'}(1)^{c}} \sum_{j \in \N} \abs{\inner{f}{R(r,f)'\psi^{q}_{j}}}^{2} \nu(df)\,  dr  \\
& = & \sum_{j \in \N} r( \psi^{q}_{j})^{2} 
= \norm{i_{r,q}}^{2}_{\mathcal{L}_{2}(\Psi_{q},\Psi_{r})}
 < \infty. 
\end{eqnarray*}
Since $(U(s,t): 0 \leq s \leq t < \infty)$ is a $(C_{0},1)$-evolution system, there exist $\vartheta_{q} \geq 0$ and a continuous seminorm $p$ on $\Psi$ such that  $ q(U(r,s)\psi) \leq e^{\vartheta_{q} (s-r)} p(\psi)$, for all $0 \leq r \leq s $, $\psi \in \Psi$. Then using the above estimates, Jensen's inequality and Fubini's theorem we have for each $\psi \in \Psi$ that
\begin{flalign*}
& \Exp \int_{0}^{t} \left(  \int_{0}^{s} \,  \int_{B_{\rho'}(1)^{c}} \, \inner{f}{R(r,f)' U(r,s) \psi} \, \nu(df)\,  dr \right)^{2} ds \\
& \leq \Exp \int_{0}^{t} \left(  s \nu \left( B_{\rho'}(1)^{c} \right) \int_{0}^{s} \,  \int_{B_{\rho'}(1)^{c}} \, \abs{\inner{f}{R(r,f)' U(r,s) \psi}}^{2} \, \nu(df)\,  dr \right) ds \\
& \leq  p(\psi)^{2} \nu \left( B_{\rho'}(1)^{c} \right) \, \Exp \int_{0}^{t} \, s e^{2 \vartheta_{q} s} \, \left(  \int_{0}^{s} \,  \int_{B_{\rho'}(1)^{c}} q' \left(  R(r,f)f \right)^{2} \,   \nu(df)\,  dr \right) ds \\
& \leq  p(\psi)^{2} \, \nu \left( B_{\rho'}(1)^{c} \right) \, \left( \Exp \int_{0}^{t} \,  \int_{B_{\rho'}(1)^{c}} q' \left(  R(r,f)f \right)^{2} \, \nu(df) dr \right) 
\left( \int_{0}^{t} \, s e^{2 \vartheta_{q} s} \,  ds \right) < \infty. 
\end{flalign*}
This proves \eqref{eqFiniSecMomeStocConvLevySquaInte}. Now we prove the existence of a continuous Hilbertian seminorm on $\Psi$ for which \eqref{eqStochConvStrongSquareMoment} is satisfied. 

Define $k: \Psi \rightarrow [0,\infty)$ by 
$$k(\psi)=\left( \Exp \int_{0}^{t} \abs{\inner{ \int^{s}_{0} \int_{\Phi'} \,  U(s,r)' R(r,f) \, L(dr, df) }{\psi} }^{2} ds \right)^{1/2}, \quad \forall \, \psi \in \Psi. $$ 
By \eqref{eqWeakStrongCompStochConvLevy} and \eqref{eqFiniSecMomeStocConvLevySquaInte} $k$ is a well-defined Hilbertian seminorm on $\Psi$. From an application of Fatou's Lemma we can check that $r$ is lower-semicontinuous and hence from Proposition 5.7 in \cite{FonsecaMora:Skorokhod} we have that $r$ is continuous on $\Psi$.  As $\Psi$ is a nuclear space, there exists a continuous Hilbertian seminorm $\varrho$ on $\Psi$ such that $k \leq \varrho$ and $i_{k,\varrho}$ is Hilbert-Schmidt. Then if $(\psi^{\varrho}_{j}: j \in \N) \subseteq \Psi$ is a complete orthonormal system in $\Psi_{\varrho}$, we have
\begin{flalign*}
& \Exp \int_{0}^{t} \, \varrho' \left( \int_{0}^{s} \int_{\Psi'} \, U(s,r)' R(r,f) \, L(dr,df) \right)^{2} ds \\
& = \Exp \int_{0}^{t}  \sum_{j \in \N} \abs{ \inner{ \int^{s}_{0} \int_{\Phi'} \,  U(s,r)' R(r,f) \, L(dr, df) }{ \psi^{\varrho}_{j} } }^{2} ds \\
& = \sum_{j \in \N} k( \psi^{\varrho}_{j})^{2} 
= \norm{i_{k,\varrho}}^{2}_{\mathcal{L}_{2}(\Psi_{\varrho},\Psi_{k})} < \infty. 
\end{flalign*} 
\end{proof}

The result of Proposition \ref{propStochConvStrongSquareMoment} together with the following result will allow us to conclude that the mild solution to \eqref{eqSEELevy} has strong second moments. 

\begin{lemma}\label{lemmStrongSecondMomentEta}
Suppose $\eta$ is square integrable. Then, for every $t>0$ there exists a continuous Hilbertian seminorm $\varrho$ on $\Psi$ such that 
\begin{equation}\label{eqStrongSecondMomentEta}
\Exp \int_{0}^{t} \, \varrho' \left( U(s,0)' \eta \right)^{2} \, ds< \infty.
\end{equation}
\end{lemma}
\begin{proof}
Let $k: \Psi \rightarrow [0,\infty)$ be given by $k(\psi)^{2}= \Exp \left( \abs{\inner{\eta}{\psi}}^{2} \right)$ for each $\psi \in \Psi$. Our hypothesis of square integrability of $\eta$ guarantees that $k$ is a Hilbertian seminorm on $\Psi$ and it is clearly continuous since $\eta$ is regular. As $\Psi$ is nuclear, similar arguments to those used in the proof of Proposition \ref{propStochConvStrongSquareMoment} shows the existence of a continuous Hilbertian seminorm $r$ on $\Psi$ such that $k \leq r$ and $\Exp \left( r'(\eta)^{2} \right)< \infty$.

Let $\vartheta_{r} \geq 0$ and $q$ a continuous seminorm  on $\Psi$ such that  $ r(U(r,s)\psi) \leq e^{\vartheta_{r} (s-r)} q(\psi)$, for all $0 \leq r \leq s $, $\psi \in \Psi$. Then, for each $\psi \in \Psi$ we have
$$ \Exp \int_{0}^{t} \, \abs{\inner{U(s,0)'\eta}{\psi}}^{2} \, ds \leq q(\psi)^{2} \, \Exp \left( r'(\eta)^{2} \right) \int_{0}^{t} s e^{2 \vartheta_{r} s} ds < \infty. $$  
Define $p: \Psi \rightarrow [0,\infty)$ by 
$$p(\psi)= \left( \Exp \int_{0}^{t} \, \abs{\inner{U(s,0)'\eta}{\psi}}^{2} \, ds \right)^{1/2}, \quad \forall \psi \in \Psi. $$
Similar arguments to those used in the proof of Proposition \ref{propStochConvStrongSquareMoment} show that $p$ is a continuous Hilbertian seminorm on $\Psi$ and that there exists a continuous Hilbertian seminorm $\varrho$ on $\Psi$, $p \leq \varrho$, such that $i_{p,\varrho}$ is Hilbert-Schmidt and that 
$$ \Exp \int_{0}^{t} \, \varrho' \left( U(s,0)' \eta \right)^{2} \, ds = \norm{i_{p,\varrho}}^{2}_{\mathcal{L}_{2}(\Psi_{\varrho},\Psi_{p})} < \infty. $$
So we have proved \eqref{eqStrongSecondMomentEta}. 
\end{proof}

Finally we have the main result of this section concerning the existence and uniqueness of a solution which is (strongly) square integrable.  

\begin{theorem}\label{theoExisUniqStronSquaIntegSolution}
Suppose that $\eta$ and the L\'{e}vy process $L$ are square integrable, and $R \in \Lambda^{2}(L,\Phi,\Psi)$. Then the mild solution $X=(X_{t}: t \geq 0)$ to \eqref{eqSEELevy} is the unique (up to modifications) weak solution to \eqref{eqSEELevy} satisfying that for every $t>0$ there exists a continuous Hilbertian seminorm $\varrho$ on $\Psi$ such that $\Exp \int_{0}^{t} \, \varrho'(X_{s})^{2} ds < \infty$. 
\end{theorem}
\begin{proof}
We already know by Theorem \ref{theoExistSoluSEELevy} that the mild solution $X=(X_{t}: t \geq 0)$  is a weak solution to \eqref{eqSEELevy}. By \eqref{eqDefMildSoluSEELevy}, Proposition \ref{propStochConvStrongSquareMoment} and Lemma \ref{lemmStrongSecondMomentEta} for every $t>0$ there exists a continuous Hilbertian seminorm $\varrho$ on $\Psi$ such that $\Exp \int_{0}^{t} \, \varrho'(X_{s})^{2} ds < \infty$. This proves the existence of a weak solution satisfying the conditions in the statement of Theorem \ref{theoExisUniqStronSquaIntegSolution}. 

To prove uniqueness, assume that $Y=(Y_{t}: t \geq 0)$ is a any weak solution to \eqref{eqSEELevy} satisfying that for every $t>0$ there exists a continuous Hilbertian seminorm $\varrho$ on $\Psi$ such that $\Exp \int_{0}^{t} \, \varrho'(Y_{s})^{2} ds < \infty$. It is clear that in such a case $Y$ has almost surely Bochner (square) integrable trajectories. Then by Theorem \ref{theoUniqueSoluSEELevy} we have that $Y$ is a version of the mild solution $X$ to \eqref{eqSEELevy}.  
\end{proof}

\subsection{Markov Property}\label{subsectMarkProp}

In this and the next section we will study properties of stochastic evolution equations of the form 
\begin{equation}\label{eqGenLanEquaLevy}
d Y_{t}=A(t)' Y_{t}dt + B' dL_{t}, \quad t \geq 0, 
\end{equation}
with initial condition $Y_{0}=\eta$ $\Prob$-a.e., for $\eta$, $A(t)$, $U(s,t)$, $L$ as in Assumption \ref{assuSEELevy}(A1)-(A3) and for $B \in \mathcal{L}(\Psi,\Phi)$. In accordance with the literature the stochastic evolution equation \eqref{eqGenLanEquaLevy} is called a \emph{(generalized) Langevin equation} with L\'{e}vy noise. 

A weak solution to \eqref{eqGenLanEquaLevy} is a $\Psi'$-valued regular adapted process $Y=(Y_{t}: t \geq 0)$ satisfying that for any given $t \geq 0$, for each $\psi \in D_{t}$ we have $\int_{0}^{t} \abs{\inner{Y_{s}}{A(s)\psi}} ds< \infty$ $\Prob$-a.e. and  
\begin{equation*}
\inner{Y_{t}}{\psi}=\inner{\eta}{\psi} +\int_{0}^{t} \inner{Y_{s}}{A(s)\psi} ds+\inner{B'L_{t}}{\psi}. 
\end{equation*}

Let $R(t,\omega,f)=B'$ for $t \geq 0$, $\omega \in \Omega$, $f \in \Phi'$. We have that $R \in \Lambda(L,\Phi,\Psi)$ since for every $T >0$ and $\psi \in \Psi$ we have
\begin{multline*}
\Exp \int_{0}^{T} \, \left[  \abs{\inner{\goth{m}}{R(s,0)'\psi}}^{2} + \mathcal{Q}(R(s,0)'\psi)^{2} + \int_{B_{\rho'}(1)} \abs{\inner{f}{R(s,f)'\psi}}^{2} \nu(df)\, \right] ds \\
\leq T \left(  \abs{\inner{\goth{m}}{B\psi}}^{2} + \mathcal{Q}(B\psi)^{2} + \rho(B\psi)^{2} \int_{B_{\rho'}(1)} \, \rho'(f)^{2} \,  \nu(df) \right)  < \infty.
\end{multline*}
Moreover, it is easy to verify that 
$\int_{0}^{t} \int_{\Phi'} \, B \psi  \,L(dr,df)=
\inner{B' L_{t}}{\psi}$ for every $\psi \in \Psi$ and hence \eqref{eqGenLanEquaLevy} can be interpreted as a particular case of \eqref{eqSEELevy} for the coefficient $R(t,\omega,f)=B'$. 
Then Theorem \ref{theoExistSoluSEELevy} shows that 
\begin{equation}\label{eqDefMildSoluSEEOU}
X_{t}=U(t,0)'\eta +\int_{0}^{t}\int_{\Phi'} \, U(t,s)' B' \, L(ds,df), \quad \forall \, t \geq 0, 
\end{equation}
is a weak solution to \eqref{eqGenLanEquaLevy}, which is known as a \emph{(generalized) Ornstein-Uhlenbeck process} with L\'{e}vy noise.

Our main objective in this section is to study the flow property and the Markov property of the mild solution \eqref{eqDefMildSoluSEEOU} to \eqref{eqGenLanEquaLevy}.
For $0 \leq s \leq t < \infty$, define $\Gamma_{s,t}: \Psi' \times \Omega \rightarrow \Psi'$ by 
$$ \Gamma_{s,t}(\psi)=U(t,s)'\psi+\int_{s}^{t}\int_{\Phi'} \, U(t,r)'B' \, L(dr,df), \quad \forall \psi \in \Psi. $$

\begin{proposition}\label{propOrnsUhleIsMarkov} Let $(X_{t}: t \geq 0)$ be the Ornstein-Uhlenbeck process. 
\begin{enumerate}
\item The family $(\Gamma_{s,t}: 0 \leq s \leq t <\infty)$ is a stochastic flow: $\Gamma_{s,s}=\mbox{Id}$ and $\Gamma_{s,t} \circ \Gamma_{r,s} = \Gamma_{r,t}$ $\forall 0 \leq r \leq s \leq t < \infty$.  
\item For every $0 \leq s \leq t < \infty$, $\int_{s}^{t} \int_{\Phi'} \, U(t,r)'B' \, L(dr,df)$ is independent of $\mathcal{F}_{s}$. 
\item For each $\psi \in \Psi$, $s,t \geq 0$, 
$$ \Exp \left( e^{i \inner{X_{s+t}}{\psi}} \, \vline \, \mathcal{F}_{s} \right) = \Exp \left( e^{i \inner{X_{s}}{U(s,s+t) \psi}}  \right) H(s,t,\psi), $$
for some deterministic complex-valued function $H(s,t,\psi)$. 
\item $(X_{t}: t \geq 0)$ is a Markov process with respect to the filtration $(\mathcal{F}_{t})$.
\end{enumerate}
\end{proposition}
\begin{proof} We adapt to our context the arguments used in the proof of  Proposition 4.1 in \cite{Applebaum:2006} and Theorem 2.2 in \cite{BojdeckiGorostiza:1991}.

To prove (1). Let $0 \leq r \leq s \leq t < \infty$. From the semigroup property of the forward evolution system $(U(t,s)': 0 \leq s \leq t< \infty)$ and Proposition \ref{propActiLineOperLevyInteg} we have
\begin{flalign*}
& \Gamma_{s,t}(\Gamma_{r,s} (\phi)) \\
& =  U(t,s)' \, \Gamma_{r,s} (\phi) + \int_{s}^{t} \int_{\Phi'} \, U(t,u)' B' \, L(du,df) \\
& =  U(t,s)' U(s,r)' \phi + U(t,s)' \int_{r}^{s} \int_{\Phi'} \, U(s,u)'B' \, L(du,df) + \int_{s}^{t} \int_{\Phi'} \, U(t,u)'B' \, L(du,df) \\
& =  U(t,r)' \phi + \int_{r}^{s} \int_{\Phi'} \, U(t,u)'B' \, L(du,df) + \int_{s}^{t} \int_{\Phi'} \, U(t,u)'B' \, L(du,df) \\
& =  U(t,r)' \phi + \int_{r}^{t} \int_{\Phi'} \, U(t,u)'B' \, L(du,df) 
= \Gamma_{r,t}(\phi). 
\end{flalign*}

To prove (2), let  $0 \leq s \leq t < \infty$. 
Observe that since $\int_{s}^{t} \int_{\Phi'} \, U(t,r)'B' \, L(dr,df)$ is regular and in view of \eqref{eqWeakStrongCompStochConvLevy},
to show that it is independent of $\mathcal{F}_{s}$ it is sufficient to check that for each $\psi \in \Psi$, 
$\int_{s}^{t} \int_{\Phi'} \, B \, U(t,r)\psi \, L(dr,df)$ is independent of $\mathcal{F}_{s}$. In effect, the first and third stochastic integrals in  the right-hand side of \eqref{eqDefiWeakStochConvLevy} are easily seem to by independent of $\mathcal{F}_{s}$ by the independent increments of $L$. The second  stochastic integral in  the right-hand side of \eqref{eqDefiWeakStochConvLevy} was defined via an It\^{o} type integration theory (see Section in \cite{FonsecaMora:2018-1}) and then standard arguments and the independent increments of $L$ show that these stochastic integrals are independent of $\mathcal{F}_{s}$. This shows (2). 

To prove (3), let $\psi \in \Psi$ and $s,t \geq 0$. 
From \eqref{eqDefMildSoluSEEOU}  we have $X_{t}=\Gamma_{0,t}(\eta)$ and then from part (1) we have $X_{t+s}=\Gamma_{s,s+t} \circ \Gamma_{0,s}(\eta)=\Gamma_{s,s+t} (X_{s})$. Using the above identity and from part (2) we have that
\begin{flalign*}
& \Exp \left( e^{i \inner{X_{s+t}}{\psi}} \, \vline \, \mathcal{F}_{s}  \right) \\
& =  \Exp \left( e^{i \inner{\Gamma_{s,s+t} (X_{s})}{\psi}}  \, \vline \, \mathcal{F}_{s} \right) \\
& =  \Exp \left( \exp \left\{ i \inner{ U(s+t,s)'X_{s}+\int_{s}^{s+t}\int_{\Phi'} \, U(s+t,r)' B' \, L(dr,df) }{\psi} \right\} \, \vline \, \mathcal{F}_{s} \right)\\
& =  \Exp \left( e^{i \inner{U(s+t,s)' X_{s}}{\psi}}  \right) \cdot \Exp \left( \exp \left\{ i \, \int_{s}^{s+t}\int_{\Phi'} \, B U(r,s+t) \psi \, L(dr,df)  \right\} \right) \\
& =  \Exp \left( e^{i \inner{X_{s}}{U(s,s+t) \psi}}  \right) H(s,t,\psi).
\end{flalign*} 
Finally, (4) is a direct consequence of (3). 
\end{proof}

\subsection{Time Regularity of Solutions}\label{subSectTimeRegu}

In this section we give sufficient conditions for the existence and uniqueness of a c\`{a}dl\`{a}g solution to stochastic evolution equations of the form \eqref{eqGenLanEquaLevy}. We will need the following result on existence of some types of stochastic convolution integrals whose paths are continuous.  

\begin{lemma}\label{lemmaDefiRandomIntegViaRegula}
Let $X=(X_{t}: t \geq 0)$ be a $\Psi'$-valued adapted process with c\`{a}dl\`{a}g paths. Consider a family $(G(s,t): 0 \leq s \leq t < \infty) \subseteq \mathcal{L}(\Psi,\Psi)$ with the property that for every $\psi \in \Psi$, $s,t \geq 0$, the mappings $[s,\infty) \ni r \mapsto G(s,r)\psi$ and $[0,t] \ni r \mapsto G(r,t)\psi$ are continuous. 

Then there exists a unique $\Psi'$-valued regular adapted process $\int_{0}^{t} G(t,s)' X_{s} ds$, $t \geq 0$, with continuous paths and   satisfying $\Prob$-a.e.  
\begin{equation}\label{eqDefiRandomIntegViaRegulari}
\inner{\int_{0}^{t} \, G(t,s)' X_{s} \, ds}{\psi}=\int_{0}^{t} \inner{X_{s}}{G(s,t)\psi} ds, \quad \forall t \geq 0, \, \psi \in \Psi. 
\end{equation}
\end{lemma}
\begin{proof}
We follow arguments used in the proof of Lemma 1.1 in \cite{FernandezGorostiza:1992}. 

Let $\psi \in \Psi$ and $t \geq 0$. Define $Y_{t}(\psi)(\omega) \defeq \int_{0}^{t} \inner{X_{s}(\omega)}{G(s,t)\psi} ds$. We will check that $Y_{t}(\psi)$ is a real-valued random variable. In effect, since $X$ is c\`{a}dl\`{a}g (for $\Prob$-a.e. $\omega \in \Omega$) we have that $s \mapsto X_{s}(\omega) \in D_{\infty}(\Psi')$ and hence from the continuity of the mapping $[0,t] \ni r \mapsto G(r,t)\psi$ we can conclude that the mapping $s \mapsto \mathbbm{1}_{[0,t]}(s) \inner{X_{s}(\omega)}{G(s,t)\psi} $ is c\`{a}dl\`{a}g on $[0,t]$. Hence, the integral $\int_{0}^{t} \inner{X_{s}(\omega)}{G(s,t)\psi} ds$ is well-defined  for $\Prob$-a.e. $\omega \in \Omega$. Moreover, Fubini's theorem shows that $\omega \mapsto Y_{t}(\psi)$ is $\mathcal{F}_{t}$-measurable. The process  $(Y_{t}: t \geq 0)$ is hence adapted. 

Now, given $\omega \in \Omega$ we have $s \mapsto X_{s}(\omega) \in D_{t}(\Psi')$ and because $\Psi$ is barrelled, there exists a continuous seminorm $p=p(t,\omega)$ on $\Psi$ such that $s \mapsto X_{s}(\omega) \in D_{t}(\Psi'_{p})$ (see Remark 3.6 in \cite{FonsecaMora:Skorokhod}). Since $[0,t] \ni r \mapsto G(r,t)\psi$ is continuous, the set $\{G(r,t)\psi: 0 \leq r \leq t \}$ is bounded in $\Psi$, hence under the seminorm $p$. Then, for each $\psi \in \Psi$, it follows that
$$ \abs{\int_{0}^{t} \inner{X_{s}(\omega)}{G(s,t)\psi} ds}  \leq  t \sup_{0 \leq s \leq t} p'(X_{s}(\omega)) \sup_{0 \leq s \leq t} p(G(s,t)).$$
We have therefore shown that for  each $t \geq 0$ and $\omega \in \Omega$, $\psi \mapsto Y_{t}(\psi)(\omega) \in \Psi'$. Hence, for each $t \geq 0$, $Y_{t}$ is a cylindrical random variable in $\Psi'$ such that the mapping $Y_{t}: \Psi \rightarrow L^{0}\ProbSpace$ is continuous. Observe moreover that for each $\omega \in \Omega$ and $\psi \in \Psi$, the mapping $t \mapsto Y_{t}(\psi)(\omega)$ is continuous. In effect, given $0 \leq r_{1} \leq r_{2}  \leq t$, we have 
\begin{flalign*}
& \abs{  \int_{0}^{r_{1}} \inner{X_{s}(\omega)}{G(s,r_{1})\psi} ds -  \int_{0}^{r_{2}} \inner{X_{s}(\omega)}{G(s,r_{2})\psi} ds}  \\
& \leq \int_{0}^{r_{1}} \abs{ \inner{X_{s}(\omega)}{G(s,r_{1})\psi- G(s,r_{2})\psi }} ds 
+ \int_{r_{1}}^{r_{2}} \inner{X_{s}(\omega)}{G(s,r_{2})\psi} ds,
\end{flalign*}
which converges to $0$ as $r_{1} \rightarrow r_{2}$ or $r_{2} \rightarrow r_{1}$ by using the strong continuity of the family $(G(s,t): 0 \leq s \leq t < \infty)$ and dominated convergence. The version of the regularization theorem for cylindrical processes in the dual of an ultrabornological nuclear space (Corollary 3.11 in \cite{FonsecaMora:2018}) shows the existence of a unique $\Psi'$-valued regular continuous process $\int_{0}^{t} G(t,s)' X_{s} ds$, $t \geq 0$, that is a version of $(Y_{t}: t \geq 0)$, i.e. satisfying \eqref{eqDefiRandomIntegViaRegulari} $\Prob$-a.e. Since $\int_{0}^{t} G(t,s)' X_{s} ds$, $t \geq 0$ is regular and weakly adapted it is (strongly) adapted. 
\end{proof}

The following theorem shows the existence of a unique c\`{a}dl\`{a}g version under the assumption that each $A(t)$  is a continuous linear operator on $\Psi$.

\begin{theorem}\label{theoCadlagVersiContiGenerator}
Suppose that $A(t) \in \mathcal{L}(\Psi,\Psi)$ for each $t \geq 0$ and the mapping $t \mapsto A(t) \psi$ is continuous from $[0,\infty)$ into $\Psi$ for every $\psi \in \Psi$. 
Then the Ornstein-Uhlenbeck process \eqref{eqDefMildSoluSEEOU}  has a unique (up to indistinguishable versions) $\Psi'$-valued regular c\`{a}dl\`{a}g version $(Z_{t}: t \geq 0)$ given by
\begin{equation} \label{eqDefiCadlagVersMildSoluLevy}
Z_{t}=U(t,0)'\eta+\int_{0}^{t} U(t,s)' A(s)' B' L_{s} ds + B'L_{t}, \quad \forall \, t \geq 0,
\end{equation} 
where $\int_{0}^{t} U(t,s)' A(t)' B' L_{s} ds$, $t \geq 0$, is a $\Psi'$-valued regular adapted process with continuous paths satisfying $\Prob$-a.e.  
\begin{equation}\label{eqDefiDetermIntegLevyMildSol}
\inner{\int_{0}^{t} U(t,s)' A(s)' B' L_{s} ds}{\psi}=\int_{0}^{t} \inner{L_{s}}{BA(s)U(s,t)\psi} ds, \quad \forall t \geq 0, \, \psi \in \Psi. 
\end{equation}
Moreover if $(L_{t}: t\geq 0)$ has continuous paths, then $(Z_{t}: t \geq 0)$ has continuous paths too. 
\end{theorem} 
\begin{proof}
We define the stochastic integral $\int_{0}^{t} U(t,s)' A(t)' B'  L_{s} ds$, by applying Lemma \ref{lemmaDefiRandomIntegViaRegula} to the $\Psi'$-valued adapted c\`{a}dl\`{a}g process $X_{t}=L_{t}$ and to the family $G(s,t)=B A(s)U(s,t)$ which by our assumptions is strongly continuous in each of its variables.  Lemma \ref{lemmaDefiRandomIntegViaRegula} then shows that   $\int_{0}^{t} U(t,s)' A(t)' B' L_{s} ds$, $t \geq 0$, is a $\Psi'$-valued regular adapted process with continuous paths satisfying \eqref{eqDefiDetermIntegLevyMildSol} $\Prob$-a.e.  

Now observe that since $B' \in \mathcal{L}(\Phi',\Psi')$, then $(B'L_{t}: t \geq 0)$ is a $\Psi'$-valued regular adapted process with c\`{a}dl\`{a}g paths. Likewise, the strong continuity of the forward evolution system $(U(t,s)': 0 \leq s \leq t)$ shows that $(U(t,0)'\eta: t \geq 0)$ is a $\Psi'$-valued regular adapted process with continuous paths. Then $(Z_{t}: t \geq 0)$ defined by \eqref{eqDefiCadlagVersMildSoluLevy} is a $\Psi'$-valued regular adapted process with c\`{a}dl\`{a}g paths. It is clear that if $(L_{t}: t\geq 0)$ has continuous paths the same is satisfied for $(B'L_{t}: t \geq 0)$ and hence $(Z_{t}: t \geq 0)$ has continuous paths by the arguments given above. 

To show that $(Z_{t}: t \geq 0)$ is a version of the Ornstein-Uhlenbeck process \eqref{eqDefMildSoluSEEOU}  it is sufficient to show that $(Z_{t}: t \geq 0)$ is a weak solution to \eqref{eqGenLanEquaLevy}. In such a case, Theorem \ref{theoUniqueSoluSEELevy} together with Proposition \ref{propCadlagImplyASBochner} shows that 
$(Z_{t}: t \geq 0)$ is the unique c\`{a}dl\`{a}g version of  \eqref{eqDefMildSoluSEEOU}. 

Let $t \geq 0$ and $\psi \in \Psi$. Using 
 \eqref{eqDefiCadlagVersMildSoluLevy}, \eqref{eqDefiDetermIntegLevyMildSol}, then \eqref{eqForwardEquation} and finally using again \eqref{eqDefiCadlagVersMildSoluLevy} we have $\Prob$-a.e.
 \begin{eqnarray*} 
\inner{Z_{t}}{\psi} & = & 
\inner{U(t,0)'\eta}{\psi}+\inner{\int_{0}^{t} U(t,s)' A(s)' B' L_{s} ds}{\psi} + \inner{B' L_{t}}{\psi}\\
& = & \inner{\eta}{U(0,t)\psi}+\int_{0}^{t}   \inner{L_{s}}{B A(s)U(s,t) \psi} ds + \inner{B' L_{t}}{\psi}\\
&= &  \inner{\eta}{\psi} + \int_{0}^{t} \inner{\eta}{U(0,r)A(r) \psi} dr + \inner{B' L_{t}}{\psi}\\
& {} & +\int_{0}^{t}  \left( \inner{L_{s}}{B A(s) \psi}
+ \int_{s}^{t} \inner{L_{s}}{B A(s)U(s,r)A(r) \psi} dr
 \right) ds   \\
& = & \inner{\eta}{\psi} + \inner{B' L_{t}}{\psi} \\
& {} & +\int_{0}^{t} \left( \inner{U(r,0)'\eta}{A(r) \psi} +\inner{L_{r}}{BA(r) \psi} + \int_{0}^{r}  \inner{L_{s}}{BA(s)U(s,r)A(r) \psi} ds \right) dr \\
& = & \inner{\eta}{\psi} + \inner{B' L_{t}}{\psi} +\int_{0}^{t} \inner{Z_{r}}{A(r)\psi} dr.
\end{eqnarray*} 
Thus $(Z_{t}: t \geq 0)$ is a weak solution to \eqref{eqGenLanEquaLevy}.  
\end{proof}

\section{Examples and Applications}\label{secExamApplica}

In this section we consider examples and applications of our theory of  stochastic evolutions equations developed in the last two sections. 

\begin{example}
Let $\Phi=\Psi=\mathcal{S}(\R)$ be the Schwartz space of rapidly decreasing functions. Let $L=(L_{t}:t\geq 0)$ be a L\'{e}vy process in the space of tempered distributions $\mathcal{S}(\R)'$ and $B \in \mathcal{L}(\mathcal{S}(\R),\mathcal{S}(\R))$. 
Consider the stochastic evolution equation on  $\mathcal{S}(\R)'$:
\begin{equation}\label{eqDilationSEELevy}
d Y_{t}=A' Y_{t}+B' dL_{t}, \quad t \geq 0,
\end{equation}
with initial condition $Y_{0}=\eta$, for a  $\mathcal{S}(\R^{d})'$-valued $\mathcal{F}_{0}$-measurable random variable, and here  $A$ is the continuous linear operator on $\mathcal{S}(\R)$ defined by $(A\psi)(x) = x \psi' (x)$, for all $x \in \R$, $\psi \in \mathcal{S}(\R)$. It is shown in Section 6 in  \cite{Babalola:1974} that $A$ is the infinitesimal generator of the $(C_{0},1)$-semigroup on $\mathcal{S}(\R)$ defined by $(S(t)\psi)(x) = \psi(e^{tx})$, for all  $t \geq 0$, $x \in \R$, $\psi \in \mathcal{S}(\R)$. 

Then Theorem \ref{theoCadlagVersiContiGenerator} shows that \eqref{eqHeatSEELevy} has a unique c\`{a}dl\`{a}g solution $(Z_{t}: t \geq 0)$ satisfying for any $t \geq 0$ and $\psi \in   \mathcal{S}(\R^{d})$:
\begin{equation}
\inner{Z_{t}}{\psi}= \inner{\eta}{\hat{\psi}_{t}} 
+\int_{0}^{t} \inner{L_{s}}{B\tilde{\psi}_{t-s})} \, ds + \inner{L_{s}}{B \psi}, 
\end{equation}
where $\hat{\psi}_{t}(x)\defeq (S(t)\psi)(x) =\psi(e^{tx})$ and $\tilde{\psi}_{t}(x) \defeq AS(t)\psi(x)=xte^{tx} \psi'(e^{tx})$ for every $t \geq 0$ and $x \in \R$. The process $(Z_{t}: t \geq 0)$ is Markov by Proposition \ref{propOrnsUhleIsMarkov}.
\end{example}

\begin{example}
Let $\Phi=\Psi=\mathcal{S}(\R^{d})$ be the Schwartz space of rapidly decreasing functions, $d \geq 1$. Let $L=(L_{t}:t\geq 0)$ be a L\'{e}vy process in the space of tempered distributions $\mathcal{S}(\R^{d})'$ and $B \in \mathcal{L}(\mathcal{S}(\R^{d}),\mathcal{S}(\R^{d}))$. 
Consider the stochastic evolution equation on  $\mathcal{S}(\R^{d})'$:
\begin{equation}\label{eqHeatSEELevy}
d Y_{t}=\Delta Y_{t}+B' dL_{t}, \quad t \geq 0,
\end{equation}
with initial condition $Y_{0}=\eta$, for a  $\mathcal{S}(\R^{d})'$-valued $\mathcal{F}_{0}$-measurable random variable, and here  $\Delta$ is the Laplace operator on $\mathcal{S}(\R^{d})'$. 

It is well-known that the Laplace operator $\Delta$ is the infinitesimal generator of the \emph{heat semigroup}  $(S(t): t \geq 0)$ which is the $C_{0}$-semigroup on $\mathcal{S}(\R^{d})$ defined as: $S(0)=I$ and for each $t>0$, 
\begin{equation*}
(S(t)\psi)(x) \defeq \frac{1}{(4 \pi t)^{d/2}}  \int_{\R^{d}} e^{-\norm{x-y}^{2}/4t} \psi(y) dy, \quad \forall \psi \in  \mathcal{S}(\R^{d}), \, x \in \R^{d}. 
\end{equation*}
If we let $\mu_{t}(x)= \frac{1}{(4 \pi t)^{d/2}}  e^{-\norm{x}^{2}/4t}$, then $\mu_{t} \in \mathcal{S}(\R^{d})$ and $S(t)\psi=\mu_{t} \ast \psi$ for each $\psi \in \mathcal{S}(\R^{d})$. 
 
The heat semigroup is equicontinuous hence a $(C_{0},1)$-semigroup. Since $\Delta \in \mathcal{L}(\mathcal{S}(\R^{d}),\mathcal{S}(\R^{d}))$, Theorem \ref{theoCadlagVersiContiGenerator} shows that \eqref{eqHeatSEELevy} has a unique c\`{a}dl\`{a}g solution $(Z_{t}: t \geq 0)$ satisfying for any $t \geq 0$ and $\psi \in   \mathcal{S}(\R^{d})$:
\begin{equation}\label{eqSoluHeatEquaLevy}
\inner{Z_{t}}{\psi}= \inner{\eta}{\mu_{t} \ast \psi} 
+\int_{0}^{t} \inner{L_{s}}{B \Delta (\mu_{t-s} \ast \psi)} \, ds + \inner{L_{s}}{B \psi}. 
\end{equation}
The existence of solutions to \eqref{eqHeatSEELevy} in the case where $L$ is a Wiener process was studied by \"{U}st\"{u}nel in \cite{Ustunel:1982-2} by following a different approach. It is worth to mention however that our solution \eqref{eqSoluHeatEquaLevy} to \eqref{eqHeatSEELevy} coincides with that obtained in \cite{Ustunel:1982-2}. 
\end{example}

\begin{example}
In this example we provide a situation in which we have a  Langevin problem defined on a Hilbert space with cylindrical L\'{e}vy noise and an appropriate nuclear space $\Phi$ can be constructed such that the Langevin problem  is solved in the dual space $\Phi'$. 

Let $(H, \inner{\cdot}{\cdot}_{H})$ be a separable Hilbert space and $-J$ be a closed densely defined self-adjoint operator on $H$ such that $\inner{-J \phi}{\phi}_{H} \leq 0$ for each $\phi \in \mbox{Dom}(J)$. Let $(T(t): t \geq 0)$ be the $C_{0}$-contraction semigroup on $H$ generated by $-J$. Assume moreover that there exists some $r_{1}$ such that $(\lambda I+ J)^{-r_{1}}$ is Hilbert-Schmidt. Given these conditions, it is known (see \cite{KallianpurXiong}, Example 1.3.2) that one can construct a   Fr\'{e}chet nuclear space $\Phi$, which is continuously embedded in $H$ and whose topology is determined by an increasing family of Hilbertian norms $\abs{\cdot}_{n}$, $n \geq 0$. This family of norms is such that  $(T(t): t \geq 0)$ restricts to a equicontinuous $C_{0}$-semigroup $(S(t): t \geq 0)$ on $\Phi$, i.e. $\abs{S(t)\phi}_{n} \leq \abs{\phi}_{n}$, $n \geq 0$. Moreover, the restriction $A$ of $-L$ to $\Phi$  is the infinitesimal generator of $(S(t): t \geq 0)$ on $\Phi$ and $A \in \mathcal{L}(\Phi,\Phi)$. 

Now, suppose that $Z=(Z_{t}: t\geq 0)$ is a cylindrical L\'{e}vy process in $H$ such that for each $t \geq 0$ the mapping $L_{t}: H \rightarrow L^{0} \ProbSpace$ is continuous. Then Example 6.2 in \cite{FonsecaMora:ReguLCS} shows that there exists a $\Phi'$-valued c\`{a}dl\`{a}g regular L\'{e}vy process $L=(L_{t}: t\geq 0)$ that is a version of $Z=(Z_{t}: t\geq 0)$.    

This way, the Langevin equation on $H$ with (cylindrical) L\'{e}vy noise,
$$ d Y_{t}=-J Y_{t} dt + dZ_{t}, \quad Y_{0}=y_{0}, $$
can be reinterpreted as the Langevin equation on $\Phi'$ with (genuine) L\'{e}vy noise,
$$ d Y_{t}=A Y_{t} dt + dL_{t}, \quad Y_{0}=y_{0}, $$
for which we know from  Theorem \ref{theoCadlagVersiContiGenerator} that a unique $\Phi'$-valued c\`{a}dl\`{a}g solution $(Z_{t}: t \geq 0)$ exists and is given by \eqref{eqDefiCadlagVersMildSoluLevy} (for $B=I$). 

The reader interested in the study of Langevin equations driven by a cylindrical L\'{e}vy process in a Hilbert space  is referred to \cite{KumarRiedle:2020} where sufficient  conditions for the existence and uniqueness of solutions are discussed. 
\end{example} 

The examples given above illustrate applications of Theorem \ref{theoCadlagVersiContiGenerator} in the case where the linear part of the stochastic evolution equation is time-homogeneous, i.e. for $A(t)=A$ $\forall t \geq 0$, where $A \in \mathcal{L}(\Psi,\Psi)$ is the generator of a $(C_{0},1)$-semigroup. 

We now discuss a class of examples in the time-inhomogeneous case under the assumption that each $A(t)$ is the generator of a $(C_{0},1)$-semigroup. We will restrict ourself to the case where  $\Psi$ is a Fr\'{e}chet nuclear space where can make usage of the theory of Kallianpur and P\'{e}rez-Abreu introduced in \cite{KallianpurPerezAbreu:1988} for the existence of a $(C_{0},1)$-backward evolution system on $\Psi$ generated by the family $(A(t): t \geq 0)$.

\begin{example}\label{examplePertubEquations}
Let  $\Psi$ be a Fr\'{e}chet nuclear space. 
Consider a family $A=(A(t):t \geq 0) \subseteq \mathcal{L}(\Psi,\Psi)$ wherein each $A(t)$ is the generator of a  $(C_{0},1)$-semigroup $(S_{t}(s): s \geq 0)$ on $\Phi$. Assume that for the family $A$ there exists an increasing sequence $(q_{n}: n \geq 0)$ of norms generating the topology on $\Psi$ such that the following two conditions hold: 
\begin{enumerate}
\item For every $k \geq 0$, there exists $m \geq k$ such that, for each $t \geq 0$, $A(t)$ has a continuous linear extension form $\Psi_{q_{m}}$ into $\Psi_{q_{k}}$ (also denoted by $A(t)$) and the mapping $t \mapsto A(t)$ is $\mathcal{L}(\Psi_{q_{m}}, \Psi_{q_{k}})$-continuous.

\item The family $A$ is \emph{stable} with  respect to $(q_{n}: n \geq 0)$, i.e. for each $T>0$ there exists $k_{0} \geq 0$ and for $k \geq k_{0}$ there are constants $M_{k}=M_{k}(T) \geq 1$ and $\sigma_{k}= \sigma_{k}(T)$ satisfying the condition:
\begin{equation}\label{eqDefiStableGenerators}
q_{k} \left(S_{t_{1}}(s_{1})S_{t_{2}}(s_{2}) \cdots S_{t_{m}}(s_{m}) \psi \right) \leq M_{k} \exp \left(  \sigma_{k} \sum_{j=1}^{m} s_{j} \right) q_{k}(\psi), \quad \forall \psi \in \Psi, \, s_{j} \geq 0, 
\end{equation} 
whenever $0 \leq t_{1} \leq t_{2} \leq \cdots \leq t_{m} \leq T$, $m \geq 0$. 
\end{enumerate}

By Theorem 1.3 in \cite{KallianpurPerezAbreu:1988} there exists a unique $(C_{0},1)$-backward evolution system $(U(s,t): 0 \leq s \leq t < \infty)$ on $\Psi$ which is generated by the family $(A(t):t \geq 0)$. 

Suppose now that we have a family $(D(t):t \geq 0) \subseteq \mathcal{L}(\Psi,\Psi)$ for which there exists $k_{0} \geq 0$  and,  for $k \geq k_{0}$ ant $t \geq 0$, $D(t)$ has a continuous linear extension form $\Psi_{q_{k}}$ into $\Psi_{q_{k}}$ (also denoted by $D(t)$) and the mapping $t \mapsto D(t)$ is $\mathcal{L}(\Psi_{q_{k}}, \Psi_{q_{k}})$-continuous.
With the same conditions as above for the family $(A_{t}:t \geq 0)$, Theorem 1.4 in \cite{KallianpurPerezAbreu:1988} shows the existence of a unique $(C_{0},1)$-backward evolution system $(V(s,t): 0 \leq s \leq t < \infty)$ on $\Psi$ which is generated by the family $(A(t)+D(t):t \geq 0)$. Moreover $(V(s,t): 0 \leq s \leq t < \infty)$ satisfies the integral equation
$$ V(s,t)\psi=U(s,t)\psi+\int_{s}^{t} \, U(s,r) D(r) V(r,t) \psi \, dr, \quad \forall t \geq 0, \, \psi \in \Psi.$$

Now if $L=(L_{t}: t \geq 0)$ is a L\'{e}vy process in a quasi-complete bornological nuclear space $\Phi$, $\eta$ is a $\mathcal{F}_{0}$-measurable random variable in $\Psi$, $B \in \mathcal{L}(\Psi,\Phi)$, with families $(A(t): t \geq 0)$ and $(D(t): t \geq 0)$ as above, the perturbed stochastic evolution equation  
\begin{equation}\label{eqGenLanEquaLevyPerturbed}
d Y_{t}=(A(t)' Y_{t}+D(t)' Y_{t})dt + B' dL_{t}, \quad t \geq 0, 
\end{equation}
with initial condition $Y_{0}=\eta$ $\Prob$-a.e., has by Theorem \ref{theoCadlagVersiContiGenerator} a unique c\`{a}dl\`{a}g weak solution $(Z_{t}: t \geq 0)$ given by 
$$ Z_{t}=V(t,0)'\eta+\int_{0}^{t} V(t,s)' (A(s)+D(s))' B' L_{s} ds + B'L_{t}, \quad \forall \, t \geq 0.$$
Furthermore, if the initial condition $\eta$ and the L\'{e}vy process $L$ are square integrable, we have by Theorem \ref{theoExisUniqStronSquaIntegSolution} that for every $t>0$ there exists a continuous Hilbertian seminorm $\varrho$ on $\Psi$ such that $\Exp \int_{0}^{t} \, \varrho'(Z_{s})^{2} ds < \infty$. 

It is worth to mention that the above result extends to the L\'{e}vy noise case the results of Theorems 2.1 and 2.3 in  \cite{KallianpurPerezAbreu:1988}, there formulated for square integrable martingale noise under the assumption that the underlying space if nuclear Fr\'{e}chet. 

For an example of families  $(A(t): t \geq 0)$ and $(D(t): t \geq 0)$ satisfying the conditions given above in this example see see Section 3 in \cite{KallianpurPerezAbreu:1988}, p.264-271. The example in \cite{KallianpurPerezAbreu:1988} points out that perturbed stochastic evolution equations of the form  \eqref{eqGenLanEquaLevyPerturbed} arise in the study of the fluctuation limit of a sequence of interacting diffusions (see \cite{HitsudaMitoma:1986, Mitoma:1985}).   
\end{example}

\section{Weak Convergence of Solutions to Stochastic Evolution Equations}\label{weakConvSEE}

In this section we give sufficient conditions for the weak convergence of a sequence of Ornstein-Uhlenbeck processes. The problem is described as follows. 

For each $n=0,1,2,\dots$, let $L^{n}=(L^{n}_{t}: t \geq 0)$ be a $\Phi'$-valued L\'{e}vy process with c\`{a}dl\`{a}g paths, $(U^{n}(s,t): 0 \leq s \leq t)$ a backward evolution system with family of generators $(A^{n}(t): t \geq 0) \subseteq  \mathcal{L}(\Psi,\Psi)$ satisfying that the mapping $t \mapsto A^{n}(t) \psi$ is continuous from $[0,\infty)$ into $\Psi$ for every $\psi \in \Psi$, $\eta^{n}$ is a $\mathcal{F}_{0}$-measurable $\Psi'$-valued regular random variable, and $B^{n} \in \mathcal{L}(\Psi,\Phi)$.

In view of Theorem \ref{theoCadlagVersiContiGenerator} the Ornstein-Uhlenbeck $X^{n}=(X^{n}_{t}: t \geq 0)$ corresponding to the generalized Langevin equation 
\begin{equation}\label{eqDefiSequenGenLangEqua}
dY^{n}_{t}=A^{n}(t)'Y^{n} dt+ (B^{n})' dL^{n}_{t}, \quad Y^{n}_{0}=\eta^{n}_{0},
\end{equation} 
has a $\Psi'$-valued c\`{a}dl\`{a}g regular version which is given by 
\begin{equation} \label{eqDefiOrnUhlCadlagWeakConverg}
X^{n}_{t}=U^{n}(t,0)'\eta^{n} +\int_{0}^{t} U^{n}(t,s)' A^{n}(s)' (B^{n})' L^{n}_{s} ds + (B^{n})' L^{n}_{t}, \quad \forall \, t \geq 0. 
\end{equation} 
Observe that because $\Psi$ is ultrabornological, Corollary 4.8 in \cite{FonsecaMora:Skorokhod} shows that for each $n=0,1,2,\dots$ the processes $ (B^{n})' L^{n}=( (B^{n})' L^{n}_{t}: t \geq 0)$ and $X^{n}=(X^{n}_{t}: t \geq 0)$ define each a random variable in $D_{\infty}(\Psi')$ with a Radon distribution. 

In the next theorem we give sufficient conditions for the weak convergence of the sequence of Ornstein-Uhlenbeck processes $X^{n}$ to $X^{0}$ in $D_{\infty}(\Psi')$. 

\begin{theorem}\label{theoWeakConvOrnsUhleProce} Assume the following:
\begin{enumerate}
\item For $n=0,1,2,\dots$, $\eta^{n}$ is independent of $L^{n}$. 
\item $\eta^{n} \Rightarrow \eta^{0}$. 
\item $U^{n}(0,t) \psi \rightarrow U^{0}(0,t) \psi$ as $n \rightarrow \infty$ for each $\psi \in \Psi$ and $t \geq 0$, and $A^{n}(s) U^{n}(s,t) \psi \rightarrow A^{0}(s) U^{0}(s,t)$ as $n \rightarrow \infty$ for each $\psi \in \Psi$ uniformly for $(s,t)$ in bounded intervals of time. 
\item $(B^{n})'L^{n} \Rightarrow (B^{0})' L^{0}$ in $D_{\infty}(\Psi')$ as $n \rightarrow \infty$. 
\end{enumerate}
Then the sequence $(X^{n}:n \in \N)$ is uniformly tight  in $D_{\infty}(\Psi')$ and $X^{n} \Rightarrow X^{0}$ in $D_{\infty}(\Psi')$ as $n \rightarrow \infty$.
\end{theorem}

For our proof of Theorem \ref{theoWeakConvOrnsUhleProce} we will require the following result which makes usage of the particular form taken by the Ornstein-Uhlenbeck process in \eqref{eqDefiOrnUhlCadlagWeakConverg}.  

Let $F:D_{\infty}(\Psi') \rightarrow D_{\infty}(\Psi')$ be defined by 
\begin{equation}\label{eqDefiTransfSoluGenerLange}
F(x)(t)= x_{t}+\int_{0}^{t} U^{0}(t,s)' A^{0}(s)' x_{s} \, ds, \quad 
\end{equation}
Observe that the arguments used in the proof of Lemma \ref{lemmaDefiRandomIntegViaRegula} guarantee that the linear mapping $F$ maps $ D_{\infty}(\Psi')$ into itself.  

\begin{lemma}\label{lemmaContiConvoMapSkorokSpace}
The mapping $F: D_{\infty}(\Psi') \rightarrow D_{\infty}(\Psi')$ defined by \eqref{eqDefiTransfSoluGenerLange} is continuous. 
\end{lemma}

Lemma \ref{lemmaContiConvoMapSkorokSpace} is an extension of Lemma 1 in  \cite{FernandezGorostiza:1992}, result originally formulated under the assumption that $\Psi$ is a nuclear Fr\'{e}chet space (see also Lemma 4.1 in \cite{PerezAbreuTudor:1992}). 
The proof of Lemma \ref{lemmaContiConvoMapSkorokSpace} can be carried out following line-by-line the arguments used in the proof of Lemma 1 in \cite{FernandezGorostiza:1992} for the family $(A^{0}(s)U^{0}(s,t): 0\leq s \leq t)$. Indeed, the only properties required on $\Psi$ is that the Banach-Steinhaus theorem hold and that each $x \in D_{T}(\Psi')$ belongs to $D_{T}(\Psi'_{p})$ for some continuous seminorm $p$ on $\Psi$, but since we are assuming that $\Psi$ is barrelled, the aforementioned properties are valid (see Theorem 11.9.1 in \cite{NariciBeckenstein} and Remark 3.6 in \cite{FonsecaMora:Skorokhod}).  The details are left to the reader.  

\begin{proof}[Proof of Theorem \ref{theoWeakConvOrnsUhleProce}]   
First, observe that since each $X^{n}_{t}$ is a regular random variable and $\Psi$ is nuclear, each $X^{n}_{t}$ has a Radon distribution (Theorem 2.10 in \cite{FonsecaMora:2018}).   In view of Theorem 6.6 in \cite{FonsecaMora:Skorokhod}  and since $\Psi$ is ultrabornological, it is sufficient to check the following:
\begin{enumerate}[label=(\alph*)]
\item For each $\psi \in \Psi$, the sequence of distributions of $\inner{X^{n}}{\phi}$ is uniformly tight on $D_{\infty}(\R)$.  
\item For all $m\in \N$, $\psi_{1}, \dots, \psi_{m} \in \Psi$, and $t_{1}, \dots, t_{m} \geq 0$, the probability
distribution of $(\inner{X^{n}_{t_{1}}}{\psi_{1}}, \dots, \inner{X^{n}_{t_{m}}}{\psi_{m}})$ converges in distribution to some
probability measure on $\R^{m}$.
\end{enumerate}

In effect, for each $n=1,2, \dots$, let $\tilde{X}^{n}$ be defined by
\begin{equation} \label{eqDefiModifiOrnUhlCadlagWeakConverg}
\tilde{X}^{n}_{t}=U^{n}(t,0)'\eta^{n}+\int_{0}^{t} U^{0}(t,s)'A^{0}(s)' (B^{n})' L^{n}_{s} ds + (B^{n})' L^{n}_{t}, \quad \forall \, t \geq 0. 
\end{equation} 
As for $X^{n}$ we have that each $\tilde{X}^{n}$ defines a $D_{\infty}(\Psi')$-valued random variable. Following the idea of  Fern\'{a}ndez and Gorostiza in \cite{FernandezGorostiza:1992} (Theorem 1), observe that in order to prove (a) and (b) it is enough to show 
\begin{enumerate} [label=(\alph*)]
\setcounter{enumi}{2}
\item $\tilde{X}^{n} \Rightarrow X^{0}$ in $D_{\infty}(\Psi')$ as $n \rightarrow \infty$, 
\item $\int_{0}^{t} \, \inner{(B^{n})' L^{n}_{s}}{(A^{n}U^{n}(s,t)\psi-A^{0}U^{0}(s,t)\psi)} \, ds \rightarrow 0$ in probability as $n \rightarrow \infty$ for each $t \geq 0$ and $\psi \in \Psi$,     
\item $d^{\infty}\left(\inner{X^{n}}{\psi}, \inner{\tilde{X}^{n}}{\psi} \right) \rightarrow 0$ in probability as $n \rightarrow \infty$ for each $\psi \in \Psi$, where $d^{\infty}(\cdot, \cdot)$ denotes the Skorokhod metric on $D_{\infty}(\R)$. 
\end{enumerate}
In effect, assume that (c), (d) and (e) holds. We start by proving (a). First, (c) shows that for each $\psi \in \Psi$ we have $\inner{\tilde{X}^{n}}{\psi} \Rightarrow \inner{X^{0}}{\psi}$ in $D_{\infty}(\R)$. Then from (e) we have $\inner{X^{n}}{\psi} \Rightarrow \inner{X^{0}}{\psi}$ in $D_{\infty}(\R)$ for each $\psi \in \Psi$, which implies (a).  

Now we prove (b). Assumption (4) in Theorem \ref{theoWeakConvOrnsUhleProce} and (d) shows that 
as $n \rightarrow \infty$, 
$$\left(\inner{X^{n}_{t_{1}}-\tilde{X}^{n}_{t_{1}}}{\psi_{1}}, \dots, \inner{X^{n}_{t_{m}}-\tilde{X}^{n}_{t_{m}}}{\psi_{m}}\right) \Rightarrow (0, \dots, 0). $$ Hence from (c) we conclude that (b) holds. 

Proof of (c): First, observe that $X^{0}=U^{0}(t,0)' \eta^{0}+ F((B^{0})'L^{0})$ and for each $n=1,2,\dots,$ we have $\tilde{X}^{n}=U^{n}(t,0)' \eta^{n}+ F((B^{n})'L^{n})$ where $F$ is the mapping defined in \eqref{eqDefiTransfSoluGenerLange}. Assumption (1) in Theorem \ref{theoWeakConvOrnsUhleProce} shows that for $\psi \in \Psi$ the real-valued processes $\inner{U^{n}(t,0)'\eta^{n}}{\psi}=\inner{\eta^{n}}{U^{n}(t,0)\psi}$ and $\inner{F((B^{n})'L^{n})(t)}{\psi}=\int_{0}^{t} \inner{(B^{n})'L^{n}_{s}}{A(s) U(s,t)\psi} ds+\inner{(B^{n})'L^{n}_{t}}{\psi}$ are independent. Hence the $\Psi'$-valued regular processes $U^{n}(t,0)'\eta^{n}$  and $F(B'L^{n})(t)$ are independent (see \cite{FonsecaMora:Levy}, Proposition 2.3). Thus, to prove (c) it suffices to show $U^{n}(\cdot,0)' \eta^{n} \Rightarrow U^{0}(\cdot,0)' \eta^{0}$ and $F((B^{n})'L^{n}) \Rightarrow F((B^{0})'L^{0})$ with convergence in $D_{\infty}(\Psi')$. 

In effect, for every $\psi \in \Psi$ assumptions (2)-(3) in Theorem \ref{theoWeakConvOrnsUhleProce} imply that  $\inner{U^{n}(\cdot,0)' \eta^{n}}{\psi} \Rightarrow \inner{U^{0}(\cdot,0)' \eta^{0}}{\psi}$ in $D_{\infty}(\R)$ and hence $(\inner{U^{n}(\cdot,0)' \eta^{n}}{\psi}: n \in \N)$ is uniformly tight in $D_{\infty}(\R)$. 
Likewise, for any choice $\psi_{1}, \dots, \psi_{m} \in \Psi$, and $t_{1}, \dots, t_{m} \geq 0$, we have as $n \rightarrow \infty$, 
\begin{equation*}
(\inner{U^{n}(0,t_{1})\eta^{n}}{\psi_{1}}, \dots, \inner{U^{n}(0,t_{m})\eta^{n}}{\psi_{m}}) \Rightarrow (\inner{U^{0}(0,t_{1})\eta^{0}}{\psi_{1}}, \dots, \inner{U^{0}(0,t_{m})\eta^{0}}{\psi_{m}}).
\end{equation*}
Then $U^{n}(\cdot,0)' \eta^{n} \Rightarrow U^{0}(\cdot,0)' \eta^{0}$ in $D_{\infty}(\Psi')$ by Theorem 6.6 in \cite{FonsecaMora:Skorokhod}. 
    
Finally, by assumption (4) in Theorem \ref{theoWeakConvOrnsUhleProce} and by Lemma \ref{lemmaContiConvoMapSkorokSpace} we have  $F((B^{n})'L^{n}) \Rightarrow F((B^{0})'L^{0})$ in $D_{\infty}(\Psi')$. This proves (c) by the arguments given above. 

Proof of (d): Assumption (4) in Theorem \ref{theoWeakConvOrnsUhleProce} shows that for each $\psi \in \Psi$ the family $(\inner{(B^{n})'L^{n}}{\psi}: n \in \N)$ is uniformly tight  in $D_{\infty}(\R)$. Since $\Psi$ is ultrabornological, Theorem 5.10 in \cite{FonsecaMora:Skorokhod} shows that the family  $((B^{n})' L^{n}: n \in \N)$ is uniformly tight  in $D_{\infty}(\Psi')$. Since $\Psi$ is barrelled, Theorem 5.2 in \cite{FonsecaMora:Skorokhod} (see the the arguments used in its proof) shows that for each $t>0$ and $\epsilon \in (0,1)$  there exists a continuous Hilbertian seminorm $q$ on $\Psi$ and $C>0$ such that  
$$ \inf_{n \geq 1} \Prob \left( \sup_{0 \leq s \leq t} q'((B^{n})'L^{n}_{s}) \leq C \right) > 1-\epsilon.   $$
Then
\begin{flalign*}
& \Prob \left( \abs{ \int_{0}^{t} \, \inner{(B^{n})' L^{n}_{s}}{(A^{n}(s)U^{n}(s,t)\psi - A^{0}(s)U^{0}(s,t)\psi)} \, ds } > \epsilon \right) \\
& \leq \Prob \left( t \, \sup_{0 \leq s \leq t} q'((B^{n})'L^{n}_{s}) \sup_{0 \leq s \leq t} q \left( A^{n}(s)U^{n}(s,t)\psi - A^{0}(s)U^{0}(s,t)\psi \right) > \epsilon \right)\\
& \leq \epsilon + \Prob \left( \sup_{0 \leq s \leq t} q \left( A^{n}(s)U^{n}(s,t)\psi - A^{0}(s)U^{0}(s,t)\psi \right) > \epsilon/Ct \right). 
\end{flalign*}
  The last term tends to 0 as $n \rightarrow \infty$ by assumption (3) in Theorem \ref{theoWeakConvOrnsUhleProce}  and since $\epsilon$ is arbitrary we have proved (d).  
  
Proof of (e): Let $\psi \in \Psi$. Since the Skorokhod topology is weaker than the local uniform topology, it suffices to show that for each $T>0$ and $\psi \in \Psi$, 
$$\sup_{0 \leq t \leq T} \abs{\inner{X^{n}_{t}-\tilde{X}^{n}_{t}}{\psi}}=\sup_{0 \leq t \leq T}  \abs{\int_{0}^{t} \, \inner{(B^{n})' L^{n}_{s}}{(A^{n}(s)U^{n}(s,t)\psi - A^{0}(s)U^{0}(s,t)\psi)} \, ds},$$
converges to $0$ in probability as $n \rightarrow \infty$. But this can be proved following the same arguments used in the proof of (d). 
\end{proof}

\begin{remark}
Theorem \ref{theoWeakConvOrnsUhleProce} constitutes an extension of Theorem 4 in \cite{FernandezGorostiza:1992} there formulated in the context of Fr\'{e}chet spaces to our more general context of quasi-complete, bornological nuclear spaces. We point out however that unlike our results in this section in \cite{FernandezGorostiza:1992} the question of existence and uniqueness of solutions to each equation  \eqref{eqDefiSequenGenLangEqua} is not addressed.

Theorem \ref{theoWeakConvOrnsUhleProce} also extends the results obtained in \cite{KallianpurPerezAbreu:1989, PerezAbreuTudor:1992}, also within the context of Fr\'{e}chet spaces. In \cite{KallianpurPerezAbreu:1989, PerezAbreuTudor:1992} the assumptions on the family of generators  are stronger than our assumptions in Theorem \ref{theoWeakConvOrnsUhleProce} and are pattered after those required for $A^{n}$ to generate a unique $U^{n}$ (as in Example \ref{examplePertubEquations}).  
\end{remark}

For each $n=0,1,2,\dots$, observe that since $(B^{n})' \in \mathcal{L}(\Phi',\Psi')$ one can easily verify that $(B^{n})'L^{n}$ is a $\Psi'$-valued L\'{e}vy process, hence possesses a family of characteristics $(\goth{m}_{n}, \mathcal{Q}_{n}, \nu_{n}, \rho_{n})$ which determine uniquely its distribution as in \eqref{levyKhintchineFormulaLevyProcess}. Given the above information, in the next result we reformulate the statement of Theorem \ref{theoWeakConvOrnsUhleProce} to replace the assumption of weak convergence of $(B^{n})'L^{n}$ to $(B^{0})' L^{0}$ in $D_{\infty}(\Psi')$ by convergence of finite-dimensional distributions and some properties of the sequence of characteristics $(\goth{m}_{n}, \mathcal{Q}_{n}, \nu_{n}, \rho_{n})$. 

\begin{theorem}\label{theoWeakConvOrnsUhleProceVerCharac} 
For each $n=0,1,2,\dots$, let $(\goth{m}_{n}, \mathcal{Q}_{n}, \nu_{n}, \rho_{n})$ be the characteristics of the $\Psi'$-valued L\'{e}vy process $(B^{n})'L^{n}$.
Assume the following:
\begin{enumerate}
\item For $n=0,1,2,\dots$, $\eta^{n}$ is independent of $L^{n}$. 
\item $\eta^{n} \Rightarrow \eta^{0}$. 
\item $U^{n}(0,t) \psi \rightarrow U^{0}(0,t) \psi$ as $n \rightarrow \infty$ for each $\psi \in \Psi$ and $t \geq 0$, and $A^{n}(s) U^{n}(s,t) \psi \rightarrow A^{0}(s) U^{0}(s,t)$ as $n \rightarrow \infty$ for each $\psi \in \Psi$ uniformly for $(s,t)$ in bounded intervals of time. 
\item There exists a continuous Hilbertian seminorm $q$ on $\Psi$ such that $\mathcal{Q}_{n} \leq q$ and $\rho_{n} \leq q$ $\forall n \in \N$, and such that the following is satisfied:
\begin{enumerate}
\item \label{condDrift} $(\goth{m}_{n}: n \in \N)$ is relatively compact in $\Psi'$,
\item \label{condGaussCova}  $\displaystyle{\sup_{n \in \N} \norm{i_{\mathcal{Q}_{n},q}}_{\mathcal{L}_{2}(\Phi_{q}, \Phi_{\mathcal{Q}_{n}})} < \infty}$, 
\item \label{condLevyMea} $\displaystyle{\sup_{n \in \N} \int_{\Phi'}(q'(f)^{2} \wedge 1) \nu_{n}(df)<\infty. }$
\end{enumerate}
\item $\forall \, m \in \N$, $\phi_{1}, \dots, \phi_{m} \in \Phi$, $t_{1}, \dots, t_{m} \in [0,T]$, 
$ (\inner{(B^{n})'L^{n}_{t_{1}}}{\phi_{1}}, \dots, \inner{(B^{n})'L^{n}_{t_{m}}}{\phi_{m}})$ converges in distribution to $ (\inner{(B^{0})'L^{0}_{t_{1}}}{\phi_{1}}, \dots, \inner{(B^{0})'L^{0}_{t_{m}}}{\phi_{m}})$.
\end{enumerate}
Then the sequences $( (B^{n})'L^{n}: n \in \N)$ and $(X^{n}:n \in \N)$ are each uniformly tight in $D_{\infty}(\Psi')$ and we have $(B^{n})' L^{n} \Rightarrow (B^{0})' L^{0}$ and $X^{n} \Rightarrow X^{0}$ in $D_{\infty}(\Psi')$ as $n \rightarrow \infty$.
\end{theorem}
\begin{proof}
By Theorem 7.2 in \cite{FonsecaMora:Skorokhod}, assumptions (4) and (5) show that $( (B^{n})' L^{n}: n \in \N)$ is uniformly tight in $D_{\infty}(\Psi')$ and that $(B^{n})' L^{n} \Rightarrow (B^{0})' L^{0}$ in $D_{\infty}(\Psi')$ as $n \rightarrow \infty$. Theorem \ref{theoWeakConvOrnsUhleProce} then shows that 
$(X^{n}:n \in \N)$ is uniformly tight in $D_{\infty}(\Psi')$ and $X^{n} \Rightarrow X^{0}$ in $D_{\infty}(\Psi')$ as $n \rightarrow \infty$.
\end{proof}


\textbf{Acknowledgements} {The author acknowledge The University of Costa Rica for providing financial support through the grant ``B9131-An\'{a}lisis Estoc\'{a}stico con Procesos Cil\'{i}ndricos''.}

\end{document}